\documentclass{article}

\usepackage{microtype}
\usepackage{booktabs} % for professional tables

\usepackage{hyperref}

\usepackage[]{algorithm}
\usepackage[noend]{algorithmic}
\usepackage{arxiv}

\usepackage[colorinlistoftodos,bordercolor=orange,backgroundcolor=orange!20,line
color=orange,textsize=scriptsize]{todonotes}
\catcode`\@=11
\catcode`\@=12

\usepackage{subcaption}

\usepackage{amsmath, amsthm, amssymb, epsfig, graphicx, float, url}

\theoremstyle{plain}
\newtheorem{theorem}{Theorem}

\newtheorem{corollary}[theorem]{Corollary}
\newtheorem{proposition}[theorem]{Proposition}

\theoremstyle{definition}

\newtheorem{remark}[theorem]{Remark}

\newcommand{\R}{\mathbb{R}}

\newcommand{\A}{\mathbb{A}}
\newcommand{\B}{\mathbb{B}}
\newcommand{\N}{\mathbb{N}}
\newcommand{\D}{\mathbb{D}}
\newcommand{\X}{\mathbb{S}}

\newcommand{\E}{\mathbb{E}}
\newcommand{\U}{\mathcal{U}}

\usepackage{bm}

\title{First-Order Methods for Wasserstein Distributionally Robust MDPs}
\author{%
   Julien Grand-Cl{\'e}ment \\
IEOR Department\\
Columbia University \\
   \texttt{julien.grand-clement@columbia.edu} \\
   \And
  Christian Kroer\\
  IEOR Department\\
Columbia University\\
  \texttt{christian.kroer@columbia.edu} \\
}

\begin{document}

\maketitle
\vspace{1cm}

\begin{abstract}
Markov decision processes (MDPs) are known to be sensitive to parameter specification.  Distributionally robust MDPs alleviate this issue by allowing for \emph{ambiguity sets} which give a set of possible distributions over parameter sets. The goal is to find an optimal policy with respect to the worst-case parameter distribution. We propose a framework for solving Distributionally robust MDPs via first-order methods, and instantiate it for several types of Wasserstein ambiguity sets. By developing efficient proximal updates, our algorithms achieve a convergence rate of $O\left(NA^{2.5}S^{3.5}\log(S)\log(\epsilon^{-1})\epsilon^{-1.5} \right)$ for the number of kernels $N$ in the support of the nominal distribution, states $S$, and actions $A$; this rate varies slightly based on the Wasserstein setup. Our dependence on $N,A$ and $S$ is significantly better than existing methods, which have a complexity of $O\left(N^{3.5}A^{3.5}S^{4.5}\log^{2}(\epsilon^{-1}) \right)$. Numerical experiments show that our algorithm is significantly more scalable than state-of-the-art approaches across several domains.
\end{abstract}

%\todo[inline]{*** Emphasize the differences from our RMDP/previous literature.}

\section{Introduction}
In many applications of sequential decision-making problems, the dynamics of the environment can only be partially modeled, because of statistical errors and inaccurate distributional information regarding the parameters of the model.
This occurs, for example, in healthcare applications~\citep{grand2020robust,steimle-2} and vehicle routing~\citep{miao2017data}.
% \citet{grand2020robust}, \citet{steimle-2} describe such an example in healthcare, and \citet{miao2017data} decribe an example in vehicle routing.
In \emph{Markov Decision Processes} (MDPs), this can be addressed using robust formulations, where the transition probabilities belong to a safety region called the \emph{uncertainty set}~\cite{Iyengar,Nilim,Kuhn,GGC}. However, robust MDPs often compute conservative policies, as they optimize only for the \textit{worst-case} kernel realization, without incorporating \textit{distributional} information about uncertainties.

\emph{Distributionally Robust MDPs} (DR-MDPs)~\citep{xu2010distributionally,yu2015distributionally} attempt to overcome the conservative nature of robust MDPs. In DR-MDPs the goal is to maximize the worst-case \emph{expected} reward, assuming that the distribution over the set of possible transition kernels is not known, but belongs to a so-called \textit{ambiguity set} consisting of all the possible measures over transition kernels. Robust MDPs can be viewed as a special case of DR-MDP, where the distribution over the set of possible kernels is restricted to Dirac masses. \citet{yang2017convex} introduces a Wasserstein ball formulation for ambiguity sets, shows the existence of an optimal policy that is Markovian, and gives a Value Iteration (VI) algorithm based on iterating a Bellman equation. %This Bellman equation requires solving a large concave maximization program.
Wasserstein distances have been shown to be particularly useful when the data is too sparse to use moment-based ambiguity sets~\citep{gao2016distributionally,esfahani2018data,zhao2018data}.
%Additionally, \citet{chen2019distributionally} show how to combine Wasserstein ambiguity set with other constraints (e.g. moment constraints) and give a conic programming  reformulation for the Bellman equation in this case. 

One drawback of the Value Iteration approach to solving DR-MDPs is that every iteration of the algorithm requires solving the associated Bellman equation. \citet{yang2017convex} shows that this Bellman equation can be reformulated as a finite-dimensional convex program with a max-min objective. 
%This still does not yield an implementable algorithm however, since that convex program is not tractable in general.
In the special case of DR-MDP policies for Wasserstein balls with a finite number of states and actions and $s$-rectangular ambiguity sets, it is possible to derive a large conic convex program using standard optimization methods.
Letting $N$ be the number of kernels in the support of the nominal distribution over the set of possible kernels, $S$ the number of states, and $A$ the number of actions of the MDP, VI with such a conic convex reformulation (solved using standard interior-point methods) returns an $\epsilon$-optimal policy in $O\left( N^{3.5} A^{3.5} S^{4.5} \log^{2}(\epsilon^{-1}) \right)$  time, for Wasserstein uncertainty based on the $\ell_2$-metric.  The same complexity results hold for Wasserstein uncertainty based on $\ell_{1}$ and $\ell_{\infty}$ metric, see end of Section \ref{sec:WDRMDP}. This time complexity is largely due to the expensive per-iteration cost of interior-point methods. This may prove prohibitively slow when the MDP instance or the number of kernels is large.

In this paper, our goal is to design algorithms based on first-order methods (FOMs), which are typically more scalable (at the cost of lower precision in the final solution).  Recently, \citet{grand2020scalable} introduced FOMs to solve \emph{robust} MDPs. Their algorithms adapt FOMs for solving static zero-sum games to the dynamic setting of MDP. Interleaving FOM updates with approximate VI updates, the authors obtain an algorithm that improves significantly on VI, in terms of dependence on $S$ and $A$, at the price of a $O(1/\epsilon)$ convergence rate rather than $O(\log(1/\epsilon))$.

\noindent
\textbf{Our contributions}

\noindent \textit{A First-Order Method for  Distributionally Robust MDP.} We build upon the Wasserstein framework for DR-MDP of \citet{yang2017convex} and on the first-order framework of \citet{grand2020scalable}.
Our algorithmic framework interleaves first-order steps and approximate Bellman updates. Our algorithm generates a sequence of iterates $(\bm{x}_1,\bm{y}_1), \ldots,(\bm{x}_T,\bm{y}_T)$, each of which is a policy $\bm{x}_t$ and an uncertainty instantiation $\bm{y}_t$. The $t$'th iterate is generated based on a first-order update on iterate $t-1$. This is achieved by computing the gradients for the first-order updates based on the linear objective arising from a value-vector estimate. By carefully interleaving approximate Bellman updates on this value-vector estimate, we show that the \textit{average} of our generated \textit{policy} iterates constructs a solution to the DR-MDP problem whose duality gap decreases at a rate of $O(1/T^{2/3})$ after $T$ first-order updates.  Note that this is different from the usual convergence guarantees for VI,  which is on the \textit{last iterate value vector}.

Our algorithmic framework attains a $O(1/T^{2/3})$ convergence rate in terms of the number of FOM steps $T$. As is expected with FOMs, this is worse than the $\log(1/\epsilon)$ rate achieved by VI. However, our dependence on  $N,A$ and $S$ is better than VI by a factor of $O(N^{2.5}AS)$.

\vspace{2mm}
\noindent \textit{Novel proximal setup.}
A crucial component in our scheme is to show that the iterate FOM updates can be computed very cheaply (in nearly linear time) for various ambiguity sets of interest. This is crucial in practice, since even a moderate number of states $S$, actions $A$, and kernels in the nominal estimate $N$, leads to a large MDP, whose instance size is $O(NAS^{2})$.
Since Wasserstein distances rely on a choice of \textit{type} and \textit{metric} (see next section), we show how to instantiate our FOM framework for several such Wasserstein ambiguity sets. We cover metrics based on the norms $\ell_1, \ell_2,$ and $\ell_\infty$, as these are the most common found in the literature on Wasserstein distances. For each of these setups, we give novel algorithms that allow the proximal first-order iterates to be computed in nearly linear time.

Combining these proximal setups with our FOM framework yields an algorithm that, to the best of our knowledge, has the best convergence rates in terms of $N, S$ and $A$ for DR-MDPs with Wasserstein balls for any of the three metrics.

\vspace{2mm}
\noindent \textit{Empirical evaluation.}  We focus our numerical experiments on $\ell_2$-based Wasserstein balls. We consider random MDPs,  and applications to machine replacement and forest management. We compare our algorithms to four state-of-the-art Value Iteration algorithms (VI, Gauss-Seidel, Anderson, and Accelerated VI) and show that our algorithm is significantly faster. Even for small instances (e.g.  $S=10, N,A=30$ or $N=10$ and $S,A=30$), our algorithm is at least twice  as fast as Value Iteration.  As instances get larger (both in terms of states/actions or number of observed kernels), our algorithm becomes much faster than all the VI variants.

\noindent
\textbf{Related works}

\noindent\textit{Faster algorithms for MDPs.} Accelerating the convergence rate of VI for regular MDPs has been studied extensively, e.g. in \citet{ref-a} and \citet{GGC-AVI}. For robust MDPs, fast Bellman updates can be computed for $s,a$-rectangular uncertainty sets \citep{Iyengar,Nilim} and $s$-rectangular uncertainty sets (see \citet{Ho} for $d_{1}$-based uncertainty set).  However, none of these algorithms extend directly to a setup with $N \geq 2$ kernels in the support of the nominal distribution, and they do not modify the Value Iteration algorithm itself. \citet{grand2020scalable} develop a FOM framework which outperforms value iteration for robust MDPs, when the size of the MDP instance is large. While this improves upon VI for large instances of \textit{robust MDPs}, their methods do no directly extend to $N \geq 2$ (i.e. to distributionally robust MDPs) nor to Wasserstein balls.  Exploiting the \textit{linear programming} formulation of \textit{non-robust} MDP,  \citet{gong2020duality} and \citet{jin2020efficiently} propose to adapt mirror descent algorithms to solve MDPs.  There is no known linear programming reformulation for \textit{robust} and \textit{distributionally robust} MDPs.
Finally, our work differs from value function approximation~\citep{tsitsiklis1997analysis,de2003linear,petrik2010optimization,tamar2014scaling} in that we can control the desired accuracy of our inexact updates, contrary to value function approximation once the basis on the chosen subspace of functions is fixed. Additionally, unlike value function approximation, our algorithm improves convergence time even when the number of states and actions remain small, if there is a large number of kernels $N$.

\noindent\textit{Distributionally Robust MDPs.} DR-MDPs were introduced in \citet{xu2010distributionally}. \citet{yu2015distributionally} considerably extend the expressiveness of the ambiguity sets (to e.g. mean absolute deviation and confidence sets) by using lifting methods developed in \citet{wiesemann2014distributionally}.
 \citet{yang2017convex} introduces Wasserstein DR-MDPs and presents a reformulation of the robust Bellman update based on Kantorovitch duality; however, the author appeals to general convex programming to solve the resulting min-max problem, which may not be tractable without exploiting further problem structure or reformulation. Our approach builds on the robust Bellman formulation of \citet{yang2017convex} by combining it with a tractable first-order setup. The authors in \citet{chen2019distributionally} combine various ambiguity sets (among others moments, $\phi$-divergences, and Wasserstein distances) and give a conic formulation for the Bellman equation for this combination of ambiguity sets.

\noindent
\textbf{Notation}
We let $P(X)$ be the set of all Borel probability measures on a set $X.$ For $n \in \N$, $\Delta(n)$ is the probability simplex of dimension $n$. For $S,A \in \N$, we let $\U=\left( \Delta(S) \right)^{A}$ be the Cartesian product of probability simplexes over states. 
%%% Local Variables:
%%% mode: plain-tex
%%% TeX-master: "AAAI"
%%% End:

% LocalWords:  healthcare

\section{Distributionally Robust MDP}\label{sec:prelim}
A Distributionally Robust MDP (DR-MDP) is a tuple $(\X,\A, \bm{c},\bm{p}_{0},\lambda,\D)$; $\X$ is the set of states and $\A$ is the set of actions. We assume a \textit{finite} set of states and actions: $|\X| = S < + \infty, |\A| = A < + \infty$.  There is a state-action cost $\bm{c} \in \R^{| \X| \times | \A| }$, an initial distribution over the set of states $\bm{p}_{0} \in \Delta(S)$ and a discount factor $\lambda$. The transition rates $(\bm{y}_{sa})_{s,a} \in (\Delta(S))^{S \times A}$ are unknown; instead, we assume that they follow a joint probability distribution $\mu$, which is known to belong to an \emph{ambiguity set} $\D$. This distribution $\mu$ is typically estimated from historical data (see next section). The goal of the decision maker is to compute a policy $\bm{x}$ in $\Pi = \left(\Delta(A) \right)^{S}$, which maps each state $s$ to a distribution  over actions, so as to minimize the worst-case infinite-horizon discounted cost, defined as $C(\bm{x},\mu) = \E_{\bm{x}} \E_{\bm{y} \sim \mu} [\sum_{t=0}^{+ \infty} \lambda^{t} c_{s_{t}a_{t}} | s_{0} \sim \bm{p}_{0} ]$. Specifically, we want to solve
\begin{equation}\label{eq:DR-MDP}
\min_{\bm{x} \in \Pi} \max_{\mu \in \D} C(\bm{x},\mu).
\end{equation}
We focus on the case of $s$\emph{-rectangular} ambiguity, where the uncertainty about transitions is independent across states. Formally,
% In this paper we focus on $s$-rectangular ambiguity sets \cite{Kuhn}, i.e. we assume
$ \D = \{ \mu \; | \; \mu = \bigotimes \mu_{s}, \mu_{s} \in \D_{s}, \forall  s \in \X\},$
where for each state $s \in \X$ the set $\D_{s}$ is a set of probability distributions over the parameters $(\bm{y}_{sa})_{a=1}^{A} \in \left( \Delta(S) \right)^{A}$ and $\bigotimes$ stands for the product over measures. % Note that $s$-rectangularity is equivalent to assuming that uncertain transition rates across different states are independent.
This is a standard assumption in the literature, as related transition rates across different states lead to intractable problems in general \cite{Kuhn}.

As detailed in \citet{yu2015distributionally} and \citet{yang2017convex}, the \textit{value vector} $\bm{v}^{*}$ of a solution $(\bm{x}^{*},\mu^{*})$ to \eqref{eq:DR-MDP} satisfies the following Bellman equation:
\begin{equation}\label{eq:Bellman-DR-MDP}
v^{*}_{s} = \min_{\bm{x}_s \in \Delta(A)} \max_{\mu_{s} \in \D_{s}} \mathop{\E}_{\bm{y}_{s} \sim \mu_{s}} \left[ \sum_{a \in \A} x_{sa} \left( c_{sa} + \lambda \bm{y}_{sa}^{\top}\bm{v}^{*} \right) \right].
\end{equation}
Moreover,  $(\bm{x}^{*},\mu^{*})$ can be recovered as the optimal solutions in the right-hand min-max problem in \eqref{eq:Bellman-DR-MDP}.
Since $(\bm{x},\bm{y}) \mapsto \sum_{a \in \A} \bm{x}_{sa} \left( c_{sa} + \lambda \bm{y}_{sa}^{\top}\bm{v}^{*} \right)$ is bilinear, the Bellman equation depends on $\mu_{s}$ only through $\E_{\bm{y}_{s} \sim \mu_{s}} \left[ \bm{y}_{s} \right]$ \cite{yu2015distributionally}. By linearity of expectation, we may maximize over the set of possible expected values for $\bm{y}_s$ instead:
\begin{equation}\label{eq:Bellman-DR-MDP-B}
v^{*}_{s} = \min_{\bm{x} \in \Delta(A)} \max_{\bm{y}_{s} \in \B_{s}} \sum_{a \in \A} x_{sa} \left( c_{sa} + \lambda \bm{y}_{sa}^{\top}\bm{v}^{*} \right),
\end{equation}
where
$\B_{s} = \{ \bm{y}_{s} \; | \; \exists \; \mu_{s} \in \D_{s} \text{ s.t. } \bm{y}_s = \E_{\bm{\hat y}_{s} \sim \mu_{s}} \left[ \bm{\hat y}_{s} \right] \}$.
%\E_{\bm{y}_{s} \sim \mu_{s}} \left[ \bm{y}_{s} \right] \; | \; \mu_{s} \in \D_{s} \}, \forall \; s \in \X.  

\subsection{Wasserstein Distributionally Robust MDP}\label{sec:WDRMDP}
We will investigate the case where the sets of densities $\D_s$ are defined by Wasserstein distances. For single-state distributionally robust optimization and chance-constrained problems, this distance has proved useful when the number of data points is too small to rely on moment estimation of the underlying distribution \citep{gao2016distributionally,esfahani2018data}. In particular, a Wasserstein ball contains both continuous and discrete distributions while balls based on $\phi$-divergences (e.g. Kullback-Leibler divergence) centered at a discrete distribution do not contain relevant continuous distributions. Additionally, $\phi$-divergences do not take into account the closeness of two distributions, contrary to Wasserstein distance. Finally, by choosing a \textit{metric} accordingly (see definition below), the Wasserstein distance can account for the underlying geometry of the space that the distributions are defined on.

Let us define Wasserstein distances and balls.
The Wasserstein distance $W_p(\mu,\nu_s)$ between two distributions $\mu$ and $\nu_s$ is defined with respect to a \emph{metric} $d$ and a type $p \in \N$ as 
\begin{align*}
W_{p}(\mu,\nu_{s}) = \min & \; \left(\E_{(x,y) \sim \kappa}  \left[ d(x,y)^{p} \right]\right)^{1/p} \\
& \kappa \in P(\U \times \U), \\
& \Pi_{1} \kappa = \mu, \Pi_{2} \kappa = \nu_{s}.
\end{align*}
%the Wasserstein distance between $\mu$ and $\nu_{s}$ with a \textit{metric} $d$ and a \textit{type} $p \geq 1$.
where $\Pi_{1} \kappa$ and $\Pi_{2} \kappa$ are the first and second marginals for a density $\kappa$ on $\U \times \U$.
%For $\kappa$ a density on $\U \times \U$, we note $\bm{x}_{1} \kappa$ the first marginal of $\kappa$ and $\bm{x}_{2} \kappa$ the second marginal of $\kappa$.
When $p \rightarrow + \infty$, we have the pointwise convergence $W_{p} \rightarrow W_{\infty}$ \cite{givens1984class} where
\begin{align*}
W_{\infty}(\mu,\nu_{s}) = \min & \; \kappa\textrm{-ess.sup}(d) \\
& \kappa \in P(\U \times \U), \\
& \Pi_{1} \kappa = \mu, \Pi_{2} \kappa = \nu_{s},
\end{align*}
with $\kappa\textrm{-ess.sup}(d)$ defined as 
\[ \inf \{ c \in \R \; | \; \kappa \left( \{ (x,y) \; | \;  d( x,y)) > c \} \right) = 0 \}.\]
 We will be interested in the norm-based metrics $d_{1}=\ell_1, d_{2}=\ell_2$ and $d_{\infty}=\ell_\infty$.

We assume that we have a nominal estimate $\nu \in \D$ of the distribution over the transition rates.  Additionally, we assume that $\nu$ has finite support, i.e. for each $s$, 
$ \nu_{s} = (1/N) \sum_{i=1}^{N} \delta_{\bm{\hat{y}}_{i,s}}. $ This occurs, for example, when $\nu$ is the empirical distribution over $N$ samples of the transition kernels, obtained from observed, historical data \cite{yang2017convex}.
%for some transition kernels $\bm{\hat{y}}_{s,i}, i=1,...,N, s \in \X.$
The ambiguity set $\D_{p,s}$ will be the set of all measures $\mu$ within some Wasserstein distance $W_p(\mu,\nu_s)$ of the nominal estimate:
\begin{equation}\label{eq:ambiguity-set-wasserstein}
\D_{p,s} = \{ \mu \in P(\U) | W_{p}(\mu,\nu_{s}) \leq \theta^{p} \}.
\end{equation}
%$
%W_{p}(\mu,\nu_{s}) = \min_{\kappa \in P(\U \times \U)} \{ \left(\E_{(x,y) \sim \kappa}  \left[ d(x,y)^{p} \right]\right)^{1/p} | \bm{x}_{1} \kappa = \mu, \bm{x}_{2} \kappa = \nu_{s} \},
%$
%$
%W_{p}(\mu,\nu_{s}) = \min_{\kappa \in P(\U \times \U)} \{ \left( \int_{(x,y) \in \U \times \U} d(x,y)^{p} d \kappa(x,y) \right)^{p} | \bm{x}_{1} \kappa = \mu, \bm{x}_{2} \kappa = \nu_{s} \},
%$
%$
%W_{\infty}(\mu,\nu_{s}) = \min_{\kappa \in P(\U \times \U)} \{   \; | \bm{x}_{1} \kappa = \mu, \bm{x}_{2} \kappa = \nu_{s} \},
%$
%\kappa-ess.sup_{\U \times \U}  \; d = \inf \{ c \in \R \; | \; \kappa \left( \{ (x,y) \; | \;  d( x,y)) > c \} \right) = 0 \}
%$.$
In a small abuse of notation, we will let $\D_{\infty,s}$ denote the Wasserstein ball \eqref{eq:ambiguity-set-wasserstein} based on $W_{\infty}$ instead of $W_{p}$, with a radius of $\theta$.
%Given a metric $d$,  the Wasserstein ambiguity sets can be reformulated as the set of expected values that can be constructed from sum of Dirac-delta functions on distributions within distance $\theta$ of the nominal point as follows~\citep{yang2017convex,bertsimas2018data,xie2020tractable}:
Given a metric $d$, and $p \in \R \bigcup \{ \infty \}$,  the set of expected kernels for the measures $\mu$ in the Wasserstein ambiguity sets $\D_{p,s}$ can be described as~\citep{yang2017convex,bertsimas2018data,xie2020tractable}:
%\begin{align*}
%\D_{p,s} & = \{ \dfrac{1}{N} \sum_{i=1}^{N} \delta_{y_{i}} \; | \; \dfrac{1}{N} \sum_{i=1}^{N} d(y_{i},\hat{y}_{i})^{p} \leq \theta^{p} \}, \\
%\D_{\infty,s} & = \{ \dfrac{1}{N} \sum_{i=1}^{N} \delta_{y_{i}} \; | \; d(y_{i},\hat{y}_{i}) \leq \theta, \forall \; i = 1, ..., N \}.
%\end{align*}
%For $p \in \R \bigcup \{ \infty \}$ the set of possible expected values is
%\begin{equation}\label{eq:v_star_def}
%v_{s}^{*} = \min_{\bm{x} \in \Delta(S)} \max_{(\bm{y}_{1}, ..., \bm{y}_{N}) \in \B_{p,s}} \dfrac{1}{N} \sum_{i=1}^{N} \LL(\bm{x},\bm{y}_{i}).
%\end{equation}
%\begin{equation}\label{eq:v_star_def}
%v_{s}^{*} = \min_{\bm{x} \in \Delta(S)} \max_{\bm{y} \in \B_{p,s}} \LL(\bm{x},\bm{y}), \forall \; s \in \X,
%\end{equation}
%where
\begin{align*}
\B_{p,s} & = \{ \dfrac{1}{N} \sum_{i=1}^{N} \bm{y}_{i}| \dfrac{1}{N} \sum_{i=1}^{N} d(\bm{y}_{i}, \bm{\hat{y}}_{i})^{p} \leq \theta^{p}, \bm{y}_{i} \in \U, \forall \; i \},  \\
\B_{\infty,s} & = \{\dfrac{1}{N} \sum_{i=1}^{N} \bm{y}_{i} | d(\bm{y}_{i}, \bm{\hat{y}}_{i}) \leq \theta, \bm{y}_{i} \in \U, \forall i=1,...,N\}.
\end{align*}
\paragraph{Computing an optimal policy}
\citet{yang2017convex} shows that for Wasserstein balls (with $p < + \infty$), there exists an optimal policy which is stationary and Markovian; we present a proof of this result for $p=+\infty$ in our Appendix \ref{app:p-infinity}. \citet{yang2017convex} also gives a \textit{Value Iteration} algorithm to compute an optimal value vector $\bm{v}^{*}$ by iterating the Bellman equation. In particular, let $F: \R^{S} \rightarrow \R^{S}$ be the Bellman operator
\begin{equation}\label{eq:Bellman-operator}
F(\bm{v})_{s}=\min_{\bm{x}_{s} \in \Delta(A)} \max_{\bm{y}_{s} \in \B_{p,s}} \sum_{a \in \A} x_{sa} \left( c_{sa} + \lambda \bm{y}_{sa}^{\top}\bm{v} \right), \forall \; s \in \X.
\end{equation}
The Value Iteration (VI) algorithm is defined as follow:
\begin{equation}\label{alg:VI}\tag{VI}
\bm{v}_{0} \in \R^{S}, \bm{v}_{\ell+1} = F(\bm{v}_{\ell}), \forall \; \ell \geq 0.
\end{equation}
$F$ is a contraction of factor $\lambda$  and
\ref{alg:VI} returns a sequence $(\bm{v}_{\ell})_{\ell \geq 0}$ such that
$\| \bm{v}^{\ell+1} - \bm{v}^{*} \|_{\infty} \leq \lambda \cdot \| \bm{v}^{\ell} - \bm{v}^{*} \|_{\infty}, \forall \; \ell \geq 0$; an $\epsilon$-optimal policy and distribution over kernels can be computed as the pair attaining the $\min \max$ in $F(\bm{v})$, if $\| \bm{v} - F(\bm{v}) \|_{\infty} < 2\lambda \epsilon (1-\lambda)^{-1}$ \citep{Kuhn}.

In Appendix \ref{app:Belman-update}, we show that \eqref{eq:Bellman-operator} can be reformulated as a convex program by invoking convex duality twice. Thus, using an Interior Point Method (IPM), $F(\bm{v})$ can be computed in $O(N^{3.5}A^{3.5}S^{3.5}\log(\epsilon^{-1}))$ arithmetic operations (\citet{BenTal-Nemirovski}, Section 4.6.1-4.6.2), for $d=d_{1}, d_{2}, d_{\infty}$. This leads to an overall complexity for Value Iteration to return an $\epsilon$-optimal policy in $O(N^{3.5}A^{3.5}S^{4.5}\log^{2}(\epsilon^{-1}))$, which can be prohibitively large when the number of kernels, states, and actions grows.
%%% mode: plain-tex
%%% TeX-master: "AAAI"
%%% End:

\section{First-Order Methods for Wasserstein DR-MDP}
Our algorithm builds upon \eqref{alg:VI}, but avoids repeatedly solving expensive convex programs. At every VI epoch $\ell \geq 1$ (we refer to VI iterations as \emph{epochs} to distinguish from FOM iterations), we have a value vector $\bm{v}^{\ell}$ and we use a FOM to compute an approximation of the Bellman update $F(\bm{v}^{\ell})$. At  VI epoch $\ell+1$, we use our approximate solution to $F(\bm{v}^{\ell})$ to warm-start the computation of an approximation to $F(\bm{v}^{\ell+1})$. We will show that the (weighted) average of the FOM strategies across \emph{all} epochs converges to a solution to the Distributionally-Robust MDP problem \eqref{eq:DR-MDP}. 

It is important to note that our scheme is very different from the following simpler approach: run \eqref{alg:VI}, but use a FOM (instead of interior point methods) to solve each of the Bellman-equation problems. This would only converge in terms of the \textit{value vector},  rather than in terms of the duality gap guarantee that we provide for the average of all pairs of policy-kernel visited (see Theorem \ref{th:conv-T}). 
In particular, our analysis allows us to construct an average of \emph{all} iterates generated across $T$ FOM iterations and allows us to use this $T$ in our convergence guarantee.  
%In contrast, a scheme based on simply running approximate VI using a FOM at each value iteration would only use the subset of iterates at the current VI iteration. 
% Nonetheless, such an approximate VI scheme could also be instantiated using our proximal setup results, which are reviewed below.

% 
First, we rewrite the strategy space for the $\bm{y}$ player to explicitly be in terms of the individual components of the averaged vector $\bm{y}_{s} = \dfrac{1}{N} \sum_{i=1}^{N} \bm{y}_{i,s}$. Concretely, we rewrite $F(\bm{v})_{s}$ from \eqref{eq:Bellman-operator} as
\begin{equation}\label{eq:Bellman-operator-hat-B}
\min_{\bm{x}_{s} \in \Delta(A)} \max_{(\bm{y}_{1,s}, ..., \bm{y}_{N,s}) \in \tilde{\B}_{p,s}} \sum_{a \in \A} \bm{x}_{sa} \left( c_{sa} + \lambda \sum_{i=1}^{N} \dfrac{1}{N} \bm{y}_{i,sa}^{\top}\bm{v} \right),
\end{equation} 
for $\tilde{\B}_{p,s}  \subset \R^{N \times S \times A}$ defined as
\begin{equation}\label{eq:def-tilde-B}
\tilde{\B}_{p,s}  = \{( \bm{y}_{i})_{i=1}^{N}| \dfrac{1}{N} \sum_{i=1}^{N} d(\bm{y}_{i}, \bm{\hat{y}}_{i})^{p} \leq \theta^{p}, \bm{y}_{i} \in \U, \forall \; i \}.
\end{equation}
As we are now considering elements indexed by $i=1, ..., N$, for the sake of conciseness we will write  $( \bm{y}_{i})_{i}$ for $(\bm{y}_{i})_{i=1}^{N}$.
This strategy space representation will be easier to design FOMs for.

\paragraph{Proximal Setup for First-Order Methods.}
Let us fix a state $s \in \X$, for which we solve~\eqref{eq:Bellman-operator}. FOMs such as the one we consider rely on having a {\em proximal setup} for the convex and compact decision spaces $\Delta(A)$ (referred to as $X$ for simplicity in this section) and $\tilde{\B}_{s} $ (referred to as $Y$).

Using $\psi_X$, we construct the \emph{Bregman divergence} $D_X$, which measures a (pseudo) distance between any pair $\bm{x},\bm{x'} \in X$ ($D_Y$ is defined analogously):
\[D_{X}(\bm{x},\bm{x'}) = \psi_{X}(\bm{x'}) - \psi_{X}(\bm{x}) - \langle \nabla \psi_{X}(\bm{x}), \bm{x'} - \bm{x} \rangle,\]

The convergence rate depends on the {\em set widths} $\Theta_{X},\Theta_{Y}$, which are the maxima of $D_{X}$ and $D_{Y}$ on $X \times X$ and $Y \times Y$.
We will also require the maximum norm-magnitude $R_X = \max_{x\in X} \|x\|_X$, 
with $R_{Y}$ defined analogously. 

We will pay particular attention to the {\em Euclidean case}, where $(\| \cdot \|_{X}, \| \cdot \|_{Y}) = (\psi_X, \psi_Y) = (\ell_{2}, \ell_{2})$, though Algorithm~\ref{alg:PD-RMDP} applies more broadly (for example, a proximal setup with the $\ell_1$ norm is also possible). The Bregman divergences are 
\begin{align}
D_{X}(\bm{x},\bm{x}') & = \dfrac{1}{2} \| \bm{x} - \bm{x}' \|_{2}^{2}, \nonumber \\
D_{Y}((\bm{y}_{i})_{i},(\bm{y}'_{i})_{i}) & = \sum_{i=1}^{N} \dfrac{1}{2 } \| \bm{y}_{i} - \bm{y}'_{i} \|_{2}^{2}. \label{eq:bregman-div-y-player}
\end{align}

Given a proximal setup, a crucial component of the FOMs we are interested in is the {\em proximal mapping}, which can effectively be thought of as a generalization of taking a step from the previous iterate in the direction of improvement along the gradient $\bm{g}$:
\begin{align*}
  \textrm{prox}_x(\bm{g}_x, \bm{x'}_s) &= \arg\min_{\bm{x}_s \in X} \langle \bm{g}_x,\bm{x} \rangle + D_X(\bm{x}_s, \bm{x'}_s),\\
  \textrm{prox}_y(\bm{g}_y, \bm{y'}_s) &= \arg \max_{\bm{y}_s \in Y} \langle \bm{g}_y ,\bm{y}_s\rangle - D_{Y}(\bm{y}_s,\bm{y'}_s).
\end{align*}

These two proximal mapping are computed once per iteration of the algorithm, with varying inputs. A crucial issue for a practical scalable method is therefore whether these proximal mappings can be computed efficiently. As we will show later, this is indeed the case for several types of distributional uncertainty that are of practical interest.
% Note that it is straightforward that $D_{X}, D_{Y}$ are the Bregman divergences related to the $\ell_{2}$ norms on $X$ and $Y$.

\noindent\textbf{Primal-Dual update for MDP.} In this paper we focus on the primal-dual FOM from \citet{ChambollePock16}, which we refer to as PDA. Given the saddle-point formulation of \eqref{eq:Bellman-operator},  for some step sizes $\tau,\sigma \in \R$ and some vector $\bm{v} \in \R^{S}$, the Primal-Dual Algorithm (PDA) repeatedly applies proximal mappings as follows:
\begin{align}
\bm{x}^{t+1}_{s} &= \textrm{prox}_x(\tau \bm{c}^{t\prime}_{s}, \bm{x}^{t}_{s}), \label{eq:prox_update_x_simple}\\
(\bm{y}^{t+1}_{i,s})_{i} &= \textrm{prox}_y(\sigma \bm{\hat h}^t_s, (\bm{y}^{t}_{i,s})_i) \label{eq:prox_update_y_simple}
% \bm{x}^{t+1}_{s}
%  & = \arg \min_{\bm{x}_{s} \in \Delta(A)} \langle\bm{x}_{s}, \bm{c'}_{s}\rangle + \dfrac{1}{2\tau} \| \bm{x}_{s} -\bm{x}^{t}_{s}\|_{2}^{2} \label{eq:prox_update_x_simple}, \\
% (\bm{y}^{t+1}_{i,s})_{i}
% & = \arg \min_{(\bm{y}_{i,s})_{i} \in \tilde{\B}_{s}} \sum_{i=1}^{N} \langle\bm{y}_{i,s}, \bm{h}_{s}\rangle + \dfrac{1}{2\sigma} \|\bm{y}_{i,s}-\bm{y}^{t}_{i,s}\|_{2}^{2},\label{eq:prox_update_y_simple}
\end{align}
where $\bm{c}^{t\prime}_{s} \in \R^{A}, c^{t\prime}_{sa} = c_{sa}+ \lambda  \dfrac{1}{N} \sum_{i=1}^{N} \bm{y}_{i,s,a}^{t \; \top}\bm{v} ,$
and $\bm{\hat h}^t_s \in \R^{N\times A \times S}, h_{ias'} = - \dfrac{\lambda}{N} (2x^{t+1}_{sa}-x_{sa}^{t})v_{s'}$ for each $i, a$ and $s'$.
After $T$ iterations, PDA obtains a $O(1/T)$ approximation to a (static) saddle-point problem such as $F(\bm{v})$~\citep{ChambollePock16}.  Various weight schemes can be chosen to accelerate the ergodic convegence \cite{GKG20}. We now show how to combine PDA updates with VI in order to compute a solution to \eqref{eq:Bellman-DR-MDP}.

\noindent\textbf{Algorithm for DR-MDP.}
Our algorithm builds upon the first-order framework introduced in \citet{grand2020scalable} for robust MDP. In particular, the horizon $T$ is divided into $k$ \textit{epochs} of lengths $1, ..., k^{2}$. During epoch $\ell$, we perform $\ell^{2}$ PDA \textit{iterations}, starting from the last policy-kernel pair computed at the previous epoch. The average of the policy-kernel pairs visited across \emph{all} epochs converges to an optimal solution of the distributionally robust MDP problem,  as shown in Theorem \ref{th:conv-T}. Our Algorithm~\ref{alg:PD-RMDP} is different from the algorithm proposed in \citet{grand2020scalable} for \textit{robust} MDP, which only optimizes for a single kernel.
This is because we  must iterate over an $N$-tuple of kernels $(\bm{y}_{1}, ..., \bm{y}_{N})$ for the max-player.
To better understand the distinction between the algorithms, note that one could apply the algorithm of \cite{grand2020scalable} directly to (\ref{eq:Bellman-operator}) since that formulation has a single $\bm{y}$. However, it is not clear how one would set up an appropriate strongly-convex function $\psi_{\B_{p,s}}$ for this space, as it suffers from degeneracy issues where the same average kernel $\bm{y}$ can be represented by multiple combinations of the samples $\bm{y}_{1}, ..., \bm{y}_{N}$.
In contrast, we will show that there are efficient proximal setups for our representation in terms of $\tilde\B_{p,s}$. Our choice of step sizes $\tau$ and $\sigma$ also specifically addresses the dimension imbalance between the $\min$-player decisions $\bm{x} \in \R^{A}$ and the $\max$-player decisions $\left(\bm{y}_{i} \right)_{i} \in \R^{NAS}$.
 
\begin{algorithm}[H]
\caption{First-order Method for Wasserstein DR-MDP}\label{alg:PD-RMDP}
\begin{algorithmic}[1]
\STATE \textbf{Input} A number of epochs $k$.\\

\STATE \textbf{Initialize} $\bm{v}^{1},\bar{\bm{x}}^{0}, \bar{\bm{y}}^{0}$ at random

\FOR{epoch $ \ell=1,..., k$}
\FOR{$s \in \X$} 
\STATE $\tau = \left( \sqrt{A} \lambda \| \bm{v}^{\ell} \|_{2} \right)^{-1}, \sigma = N \sqrt{A} (\lambda \| \bm{v}^{\ell} \|_{2})^{-1}$
\STATE $\tau_{\ell} = \sum_{k'=1}^{(\ell-1)^{2}}k'$
%, compute \[(\bm{x}_{\tau_{\ell}+1},(\bm{y}_{\tau_{\ell}+1,i})_{i}), ..., (\bm{x}_{\tau_{\ell}+\ell^{2}},(\bm{y}_{\tau_{\ell}+\ell^{2},i})_{i})\] by running PDA for $\ell^{2}$ iterations with step sizes $\tau, \sigma$, with $\bm{v} = \bm{v}^{\ell}$ and warm-started at $(\bm{x}_{\tau_{\ell-1}},(\bm{y}_{\tau_{\ell-1},i})_{i})$. \label{alg:PD-T-ell-times}
\FOR{$t = \tau_\ell, \ldots, \tau_\ell+T_\ell$}
\STATE $\bm{x}_s^{t+1} = \textrm{prox}_x(\tau\bm{c}^{t\prime}_{s}, \bm{x}^t_s)$
\STATE $(\bm{y}^{t+1}_{i,s})_{i} = \textrm{prox}_y(\sigma\bm{\hat h}^t_s, (\bm{y}^{t}_{i,s})_i)$
\ENDFOR
% \STATE Apply $\ell^2$ iterations of PDA with stepsizes $\tau,\sigma$ on $\bm{v}=\bm{v}^\ell$ to generate the iterates $\{(\bm{x}_{t},(\bm{y}_{t,i})_{i})\}_{t=\tau_\ell+1}^{\tau_\ell+\ell^2}$
\STATE $S_\ell = \sum_{t = \tau_\ell}^{\tau_{\ell}+\ell^2} t$ \label{alg:normalization const}
\STATE $(\bar{\bm{x}}^{\ell}_{s}, (\bar{\bm{y}}^{\ell}_{i,s})_{i}) = \sum_{t=(\ell+1)}^{\tau_\ell+\ell^2} \frac{t}{S_\ell} (\bm{x}_{t},(\bm{y}_{t,i})_{i})$
% the weighted averages for the iterates \[(\bm{x}_{\tau_{\ell}+1},(\bm{y}_{\tau_{\ell}+1,i})_{i}), ..., (\bm{x}_{\tau_{\ell}+\ell^{2}},(\bm{y}_{\tau_{\ell}+\ell^{2},i})_{i}),\] with weights $\omega_{\tau_{\ell}+1},...,\omega_{\tau_{\ell} + \ell^{2}}$.
\label{alg:PD-compute-averages}
\STATE Compute $ \bar{\bm{y}}^{\ell}_{s} \in \B_{s}$ as $\bar{\bm{y}}^{\ell}_{s} = \dfrac{1}{N} \sum_{i=1}^{N}  \bar{\bm{y}}^{\ell}_{i,s}$
\STATE  Update $v^{\ell+1}_{s} = F^{\bar{\bm{x}}^{\ell}_{s},\bar{\bm{y}}^{\ell}_{s}}(\bm{v}^{\ell})_{s}$ \label{alg:PD-udate-v-ell}
	 \ENDFOR
\ENDFOR
\STATE Let $S_T = \sum_{t = 1}^{T} t$
\STATE \textbf{Output} $(\bar{\bm{x}}^{T}_{s}, (\bar{\bm{y}}^{T}_{i,s})_{i}) = \sum_{t=1}^{T} \frac{t}{S_T} (\bm{x}_{t},(\bm{y}_{t,i})_{i})$
\end{algorithmic}
\end{algorithm}
Algorithm~\ref{alg:PD-RMDP} guarantees a bound on the \textit{duality gap} of a policy-kernel pair $(\bm{x}, \bm{y})$ defined as 
\begin{equation}\label{eq:duality-gap-in-bellman}
\max_{s \in \X} \{ \max_{\bm{y'} \in \B_{s}} F^{\bm{x}, \bm{y'}}(\bm{v}^{*})_{s} - \min_{\bm{x'} \in \Delta(A)} F^{\bm{x'}, \bm{y}}(\bm{v}^{*})_{s} \},
\end{equation}
where
$ F^{\bm{x}, \bm{y}}(\bm{v})_{s}=\sum_{a \in \A} x_{sa} \left( c_{sa} + \lambda \bm{y}_{sa}^{\top}\bm{v} \right).$ Note that \eqref{eq:duality-gap-in-bellman} $\leq \epsilon/2$ guarantees that $\bm{x}$ is a $\epsilon$-optimal policy in \eqref{eq:DR-MDP}.
We give a detailed proof of our theorem in Appendix \ref{app:proof-main-th}.
\begin{theorem}\label{th:conv-T}
Let $\bm{v}^{*}$ be the value vector for a pair $\bm{x}^{*},\bm{y}^{*}$ of optimal solutions to the Bellman equation \eqref{eq:DR-MDP}.
Let $\bar{\bm{x}}^{T},\bar{\bm{y}}^{T}$ the output of Algorithm \ref{alg:PD-RMDP} after $T$ iterations.
% Let $\bar{\bm{x}}^{T},\bar{\bm{y}}^{T}$ be the averages of the iterates visited by our algorithm for the weights $1, ..., T$, with $T=1 + ... + k^{2}$ the number of PDA updates.

Then the duality gap \eqref{eq:duality-gap-in-bellman} of $\bar{\bm{x}}^{T},\bar{\bm{y}}^{T}$ is upper bounded by $ O\left( \dfrac{\sqrt{S}}{\sqrt{N}}R_{X}R_{Y}\left( \dfrac{\Theta_{X}}{\tau} + \dfrac{\Theta_{Y}}{\sigma} \right) \dfrac{1}{T^{2/3} }\right).$
\end{theorem}

Therefore, Algorithm~\ref{alg:PD-RMDP} returns a sequence of policies which converges to an optimal solution to the Distributionally Robust MDP over Wasserstein balls. In order to give the number of arithmetic operations for Algorithm \ref{alg:PD-RMDP} before returning an $\epsilon$-optimal policy, there remains to investigate the complexity of the proximal updates \eqref{eq:prox_update_x_simple}-\eqref{eq:prox_update_y_simple}.

%\begin{remark} 
% In Algorithm \ref{alg:PD-RMDP}, we choose period lengths of $T_{\ell}=\ell^{2}$ at period $\ell$ and linear weight $\omega_{t} = t$. Note that Algorithm \ref{alg:PD-RMDP} can be tuned with any period length and increasing weight schemes; for $T_{\ell} = \ell^{q}, q \geq 0,$ the theoretical convergence rate of Algorithm \ref{alg:PD-RMDP} is $O \left( C/T^{q/(q+1)} \right)$ (with $C=(\sqrt{S/N}) \cdot R_{X}R_{Y}/(\Theta_{X}/\tau + \Theta_{Y}/\sigma)$), independent of the weighting scheme and approaching $O \left( C/T \right)$ as $q$ becomes larger.
% \end{remark}

\begin{remark} We could use other FOMs than PDA in Algorithm~\ref{alg:PD-RMDP}. For example, Mirror Prox would yield a similar rate~\cite{nemirovski2004prox}, while Mirror Descent would yield a slower rate.
 However, because the objective for our FOMs is bilinear and not \textit{strongly} convex-concave, it is not clear that we can use accelerated FOMs (e.g., Section 5 and Section 6 in \cite{ChambollePock16} to obtain a linear convergence rate.
\end{remark}

\section{Convergence rate for Wasserstein balls.}
Note that in Theorem \ref{th:conv-T}, we only provide a convergence rates in term of the number of PD iterations $T$. In order to obtain our complexity results, we now turn to investigating the complexity of the primal-dual updates \eqref{eq:prox_update_x_simple} and \eqref{eq:prox_update_y_simple}. The uncertainty set $\tilde{\B}_{p,s}$ is quite unusual in the first-order methods literature, where most of the updates are computed in closed-form upon the simplex or the non-negative orthant. One of the main contributions of this paper is to design novel efficient algorithms for computing \eqref{eq:prox_update_y_simple} when the metric $d$ is $d_{1}, d_{2}$ or $d_{\infty}$. In particular in Proposition \ref{prop:finite-type} we show that we can compute \eqref{eq:prox_update_y_simple} in nearly linear time. To the best of our knowledge, we are the first to present efficient algorithms for computing the proximal updates on intersection of simplices and (various) Wasserstein balls.

\noindent\textbf{Proximal setup for $\bm{x}$ player}
The proximal update for the $\bm{x}$ player \eqref{eq:prox_update_x_simple} is the classical proximal update onto the simplex of dimension $A$, and can be computed in $O(A\log(A))$ operations \citep{BenTal-Nemirovski}.

\noindent\textbf{Proximal setup for $\bm{y}$ player}
% Our goal in this section is to develop efficient algorithms for the proximal update for the $\bm{y}$ player \eqref{eq:prox_update_y_simple}. Here we present our results for the type-p Wasserstein balls, for $p<+\infty$; we extend our results to $p=\infty$ in the appendices.
Since \eqref{eq:prox_update_y_simple} decomposes into independent problems for each state, we drop the index $s$ in our formulation of \eqref{eq:prox_update_y_simple} and assume that we are solving for some arbitrary state $s$.
For $p < + \infty$, the proximal update of the max player  \eqref{eq:prox_update_y_simple} from a kernel $\bm{y'}$ can be reformulated as
\begin{equation}\label{eq:prox-update-min-player-p}
\begin{aligned}
\min & \; \sum_{i=1}^{N} \langle \bm{y}_{i} , \bm{h} \rangle + \dfrac{1}{2 \sigma} \|\bm{y}_{i} - \bm{y'}_{i} \|_{2}^{2} \\
& \bm{y}_{1}, ..., \bm{y}_{N} \in \U, \\
&  \dfrac{1}{N} \sum_{i=1}^{N} d(\bm{y}_{i}, \bm{\hat{y}}_{i})^{p} \leq \theta^{p}.
\end{aligned}
\end{equation}
In the next propositions, we show that \eqref{eq:prox-update-min-player-p} can be solved efficiently, for $d$ equal to $d_{1}, d_{2}$ and $d_{\infty}$. The proof for each case is different, but follows a similar argument: 
\begin{enumerate}
\item We first introduce a Lagrange multiplier $\gamma$ for the last constraint. This simplifies the problem of computing \eqref{eq:prox-update-min-player-p} to solving $N$ sub-problems over $\U$, each of the form
\begin{equation}\label{eq:prox-update-min-player-p-lagrange}
\begin{aligned}
\min & \; \langle \bm{y}_{i} , \bm{h} \rangle + \dfrac{1}{2 \sigma}  \|\bm{y}_{i} - \bm{y'}_{i} \|_{2}^{2} +\gamma  \cdot d(\bm{y}_{i}, \bm{\hat{y}}_{i})^{p}  \\
& \bm{y}_{i} \in \U.
\end{aligned}
\end{equation}
\item
We then turn to efficiently solving \eqref{eq:prox-update-min-player-p-lagrange}.
\begin{itemize}
\item 
For $d=d_{2}, p=2$, \eqref{eq:prox-update-min-player-p-lagrange} can be rewritten as a series of Euclidean projections onto the simplex $\Delta(S)$, as $\U = \left( \Delta(S)\right)^{A}$.
\item For $d=d_{1}, p=1$, we introduce Lagrange multipliers $\alpha_{i,s,a}$ for each simplex constraint $\bm{y}_{i,s,a}^{\top}\bm{e} =1$; we can then solve the resulting problems using the KKT conditions.  By carefully inspecting the breakpoints of the Lagrangian for the multipliers $\alpha_{i,s,a}$, we do not need to use bisection to find the multipliers $\alpha_{i,s,a}$; see Appendix \ref{app:computing-FOM}.
\item Finally, for $d=d_{\infty}, p=1$, we use bisection to find an optimal $\alpha$ such that  $d(\bm{y}_{a}, \bm{\hat{y}}_{i,a}) \leq \alpha, $ for all $ a \in \A$. Then we solve the problem of Euclidean projection onto the simplex $\Delta(S)$ with box constraints.
\end{itemize}
\item Having designed efficient algorithms for solving \eqref{eq:prox-update-min-player-p-lagrange}, we use a bisection method on the multiplier $\mu$ and return an optimal solution of \eqref{eq:prox-update-min-player-p}.
\end{enumerate}
Summarizing the above ideas, we have the following proposition. We present the detailed proof in Appendix \ref{app:computing-FOM}.
%\begin{small}
%\begin{table}[h!]
%\centering
%\begin{tabular}{||c| c|c| c||} 
% \hline
% Type & Metric & Prox. setup & Complexity  \\ [0.5ex] 
% \hline\hline
%p=2 & $d_{2}$ &  $\ell_{1}$ & $NA^{4}S^{2}\log(S)\log^{2}(\epsilon^{-1})\epsilon^{-1}$\\
% p=2 & $d_{2}$ &  $ \ell_{2}$ & $NA^{2}S^{3}\log(S)\log(\epsilon^{-1})\epsilon^{-1}$ \\
%p=1 & $d_{1}$ &  $ \ell_{2}$  & $NA^{2}S^{3}\log(S)\log(\epsilon^{-1})\epsilon^{-1}$ \\ 
%p=1 & $d_{\infty}$ &  $ \ell_{2}$ & $NA^{2}S^{2}F(S,\epsilon)\log^{2}(\epsilon^{-1})\epsilon^{-1}$  \\
%p=$\infty$ & $d_{2}$ &  $ \ell_{1}$ & $NA^{4}S^{2}\log(S)\log^{2}(\epsilon^{-1})\epsilon^{-1}$\\
% p=$\infty$ & $d_{2}$ &  $\ell_{2}$ & $NA^{2}S^{3}\log(S)\log(\epsilon^{-1})\epsilon^{-1}$ \\
%p=$\infty$ & $d_{1}$ &  $ \ell_{2}$ & $NA^{2}S^{3}\log(S)\log(\epsilon^{-1})\epsilon^{-1}$ \\ 
%p=$\infty$ & $d_{\infty}$ &  $\ell_{2}$ & $NA^{2}S^{2}F(S,\epsilon)\log^{2}(\epsilon^{-1})\epsilon^{-1}$  \\ \hline
%\end{tabular}
%\caption{Complexity results presented in this paper. This is to be compared to $O\left(N^{3.5}A^{3.5}S^{4.5}\log(\epsilon^{-1}) \right)$ for state-of-the-art approaches.}
%\label{table:1}
%\end{table}
%\end{small}
 \begin{proposition}\label{prop:finite-type}
 % \begin{enumerate}
 % \item
   Let $d=d_{2},p=2$ or $d=d_{1},p=1$.
The proximal update \eqref{eq:prox-update-min-player-p} can be computed in $O \left(NA S \log(S) \log(\epsilon^{-1}) \right)$ arithmetic operations.

% \item
%  Let $d=d_{1}$ and $p=1$.
%The proximal update \eqref{eq:prox-update-min-player-p} can be computed in $O \left(NA S \log(S) \log(\epsilon^{-1}) \right)$ arithmetic operations.
% \item 

Let $d= d_{\infty}, p=1$.
The proximal update \eqref{eq:prox-update-min-player-p} can be computed in $O \left(NA S \log(S)\log^{3}(\epsilon^{-1}) \right)$ arithmetic operations.
 % \end{enumerate}
\end{proposition}
We can now give the overall convergence rates of our algorithms in the following theorem.
\begin{theorem}\label{th:conv-rates}
%Let 
%\begin{align}
%\textsc{comp}_{2} & = O \left(NA^{2.5}S^{3.5}\log(S)\log(\epsilon^{-1})\epsilon^{-1.5} \right) \label{eq:comp_ell_2_setup}.
%\end{align}
  The total number of arithmetic operations needed to compute an $\epsilon$-optimal solution to the Distributionally Robust MDP problem \eqref{eq:DR-MDP} using Algorithm \ref{alg:PD-RMDP} is $O \left(NA^{2.5}S^{3.5}\log(S)\log^{m}(\epsilon^{-1})\epsilon^{-1.5} \right)$,
where $m=1$ for $d=d_{2}$ and $p \in \{2,+\infty\}$, $d=d_{1}$ and $p \in \{1,+\infty\}$, and $m=3$ for $d=d_{\infty}$ and $p \in \{1,+\infty\}$.
%where $m=1$ for
%\begin{itemize}
%\item $d=d_{2}$ and $p \in \{2,+\infty\}$,
%\item $d=d_{1}$ and $p \in \{1,+\infty\}$,
%\end{itemize}
%and $m=3$ for $d=d_{\infty}$ and $p \in \{1,+\infty\}$,
\end{theorem}
\begin{proof}
We show here our proof for $d=d_{2}$ and $p \in \{2,+\infty\}$,
and $d=d_{1}$ and $p \in \{1,+\infty\}$; the proof for $d=d_{\infty}$ and $p \in \{1,+\infty\}$, follows the same argument.

For our choice of $\| \cdot \|_{X}, \| \cdot \|_{Y}$, Bregman divergences and step sizes we have (see \citet{BenTal-Nemirovski})
\begin{itemize}
\item $R_{X} = O(1), R_{Y} = O(\sqrt{NA})$,  
\item $\Theta_{X} = O(1), \Theta_{Y} = O(NA)$,
\item $\Theta_{X}/\tau =\Theta_{Y}/\sigma= \sqrt{A} \lambda \| \bm{v}^{\ell} \|_{2} = O \left( \sqrt{AS} \right),$
%\item $ \Theta_{Y}/\sigma =  \Theta_{X}/\tau = O \left( \sqrt{AS} \right),$
\end{itemize}
where we have used the norm equivalence between $\| \cdot \|_{2}$ and $\| \cdot \|_{\infty}$ in $\R^{S}$ in the last two lines.
%We have
%
%\[ \tau = \left( \sqrt{A} \lambda \| \bm{v}^{\ell} \|_{2} \right)^{-1}, \sigma = N \sqrt{A} (\lambda \| \bm{v}^{\ell} \|_{2})^{-1}. \]  Note that $\bm{v}^{\ell}$ is the cost associated with playing policies-kernels \[\bar{\bm{x}}^{\ell}_{s}, (\bar{\bm{y}}^{\ell}_{s,i})_{i}, ...,\bar{\bm{x}}^{0}_{s}, (\bar{\bm{y}}^{0}_{s,i})_{i}\]
%for $\ell$ periods. Therefore, using the equivalence of $\| \cdot \|_{2}$ and $\| \cdot \|_{\infty}$ it is straightforward that
%\[ \dfrac{1}{\sigma} \leq \dfrac{\lambda \left( \max_{s,a} c_{s,a} \right) \sqrt{S}}{(1-\lambda)\sqrt{N}}.\]
Therefore following Theorem \ref{th:conv-T} we have that the duality gap \eqref{eq:duality-gap-in-bellman} of the policy returned by Algorithm \ref{alg:PD-RMDP} after $T$ PD iterations is bounded above by 
$O \left(  \dfrac{SA}{T^{2/3}} \right)$.
Each PD iteration for these choices of $d$ and $p$ can be computed in 
$O \left(NA S \log(S) \log(\epsilon^{-1}) \right).$
Note that we have to compute PD iterations for each state $s \in \X$; therefore, Algorithm \ref{alg:PD-RMDP} returns an $\epsilon$-optimal policy to the Distributionally Robust MDP problem in 
$O \left(N A^{2.5} S^{3.5}\log(S)\log(\epsilon^{-1})\epsilon^{-1.5} \right)$.
\end{proof}
Comparing Algorithm \ref{alg:PD-RMDP} to Value Iteration, we improve upon the dependence on the problem size by a factor of $O(N^{2.5}AS)$, at the cost of a $\epsilon^{-1.5}$ convergence rate in terms of the accuracy $\epsilon$. The improvement in terms of $N$ is better than in terms of $S$ and $A$ because the number of kernels $N$ only plays a role for the max-player; this is also the reason why we choose different step sizes $\tau$ and $\sigma$ in Algorithm \ref{alg:PD-RMDP}. 
%We show in our numerical experiments of the next section that when the numbers of states and kernels increase, our algorithm may converge significantly faster than \ref{alg:VI}.
\begin{remark}[Epoch and weight scheme]
The above results are for epoch lengths $T_\ell = \ell^2$. By choosing larger values $T_{\ell} = \ell^{q}$ where $q$ tends to infinity, our algorithm approaches a complexity of $O \left(NA^{2}S^{3}\log(S)\log^{m}(\epsilon^{-1})\epsilon^{-1} \right)$. Thus it is possible to improve upon VI by a total factor of $O(N^{2.5}A^{1.5}S^{1.5})$ by choosing a large $q$. 
%Practically speaking, we will see in the numerical section that $q=2$ already leads to substantial scalability increases; analyzing our first-order algorithmic framework for various choices of epoch schemes (e.g. non-stationary choices of $T_{\ell}$) could improve the empirical performances even further. 
Additionally, we have presented Algorithm \ref{alg:PD-RMDP} with \textit{linear} weights, i.e. the weight is $t$ for the iterate $(\bm{x}^{t},(\bm{y}_{i}^{t})_{i})$. Note that Algorithm \ref{alg:PD-RMDP} can be implemented with any (increasing) weight schemes; we found that for a weight scheme of $t^{p},  p \geq 0$, the convergence rate of Algorithm \ref{alg:PD-RMDP} \textit{does not depend of $p$}, even though numerically,  $p=1$ performs better than $p=0$.
\end{remark}

\section{Numerical Experiments}
\begin{figure*}[htp]
  \begin{subfigure}{0.32\textwidth}
\centering
         \includegraphics[width=1.0\linewidth]{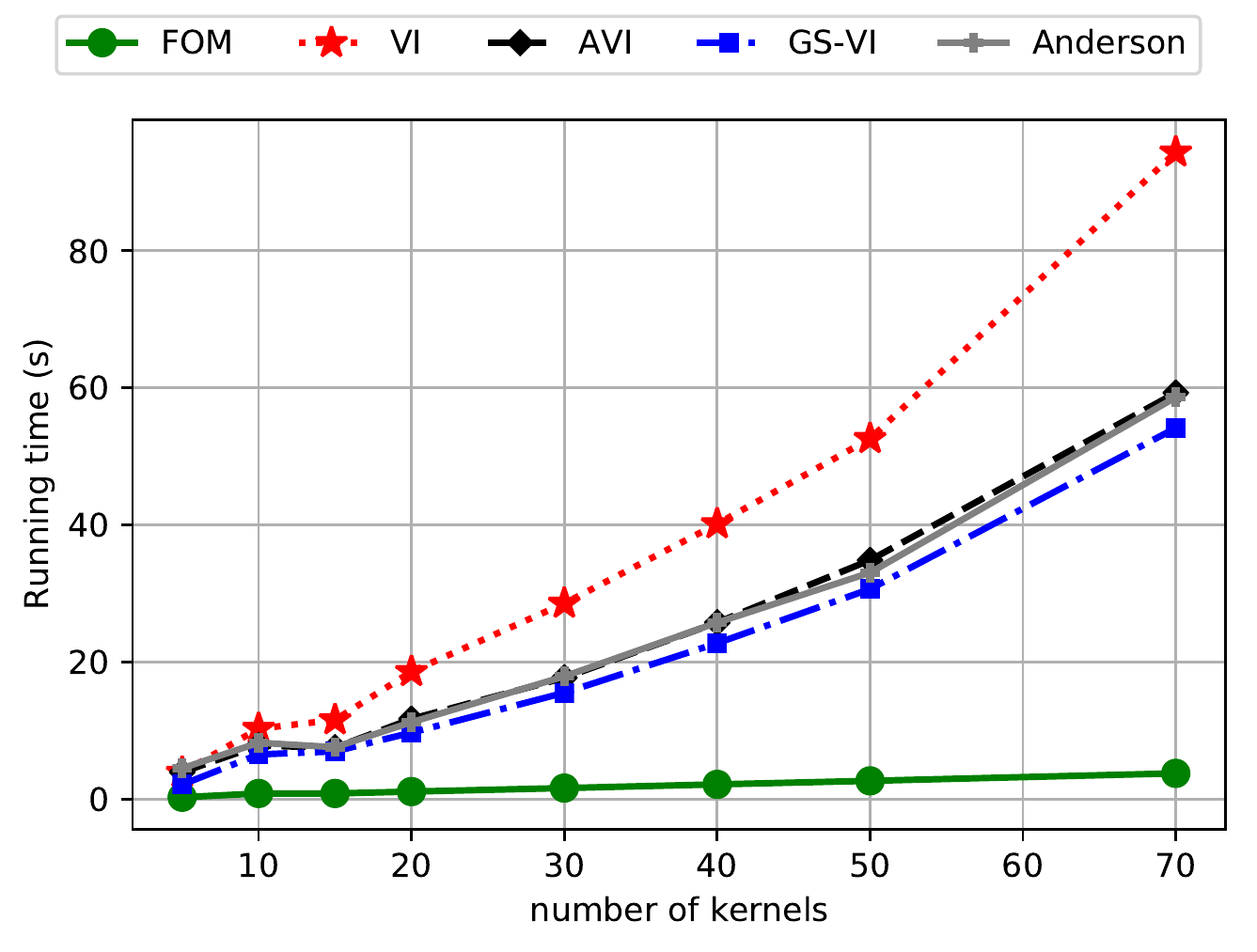}
         \caption{Forest.}
                    \label{fig:kernel_forest}
  \end{subfigure}
  \begin{subfigure}{0.32\textwidth}
\centering
         \includegraphics[width=1.0\linewidth]{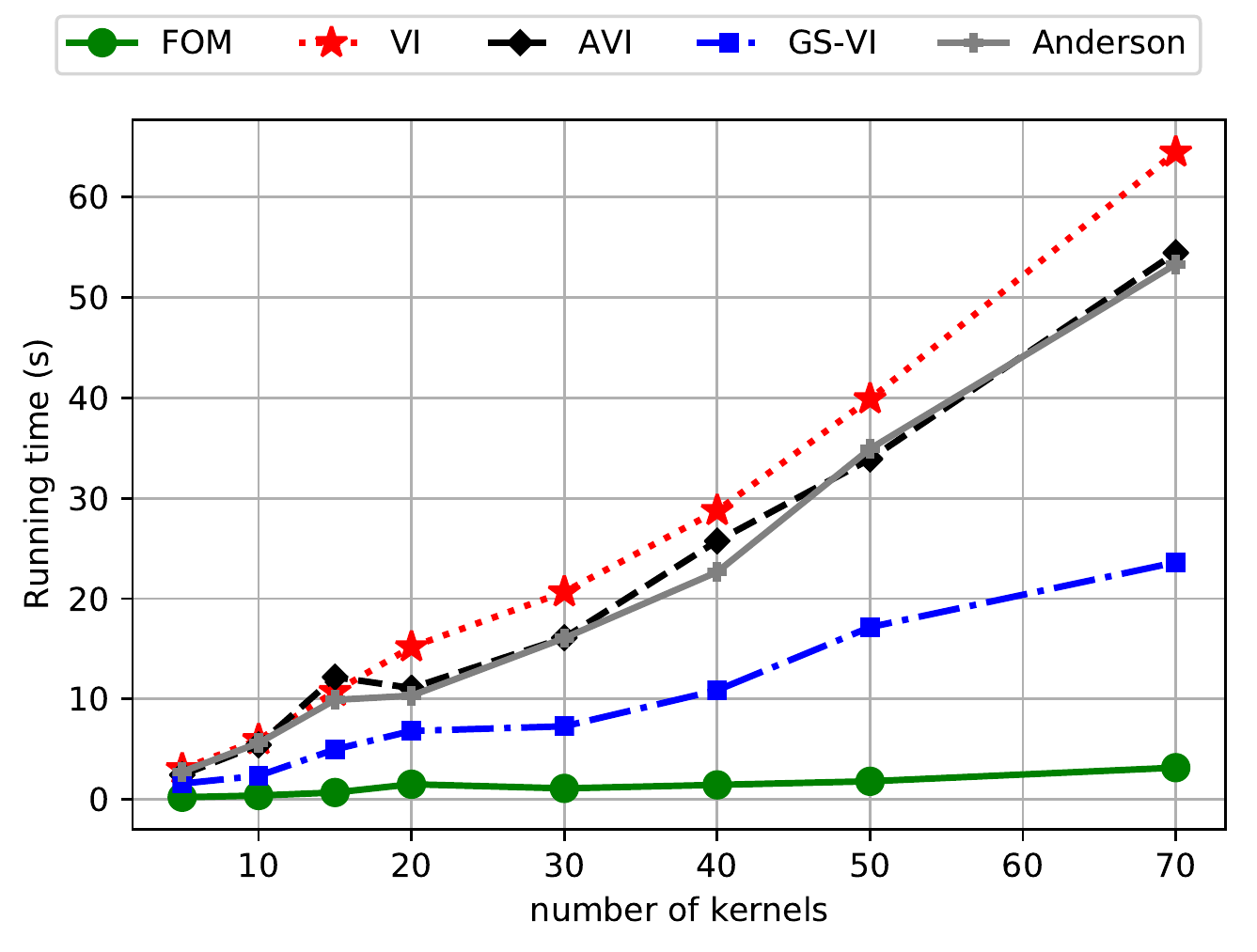}
         \caption{Machine.} 
                 \label{fig:kernel_machine}
  \end{subfigure}
\begin{subfigure}{0.32\textwidth}
\centering
         \includegraphics[width=1.0\linewidth]{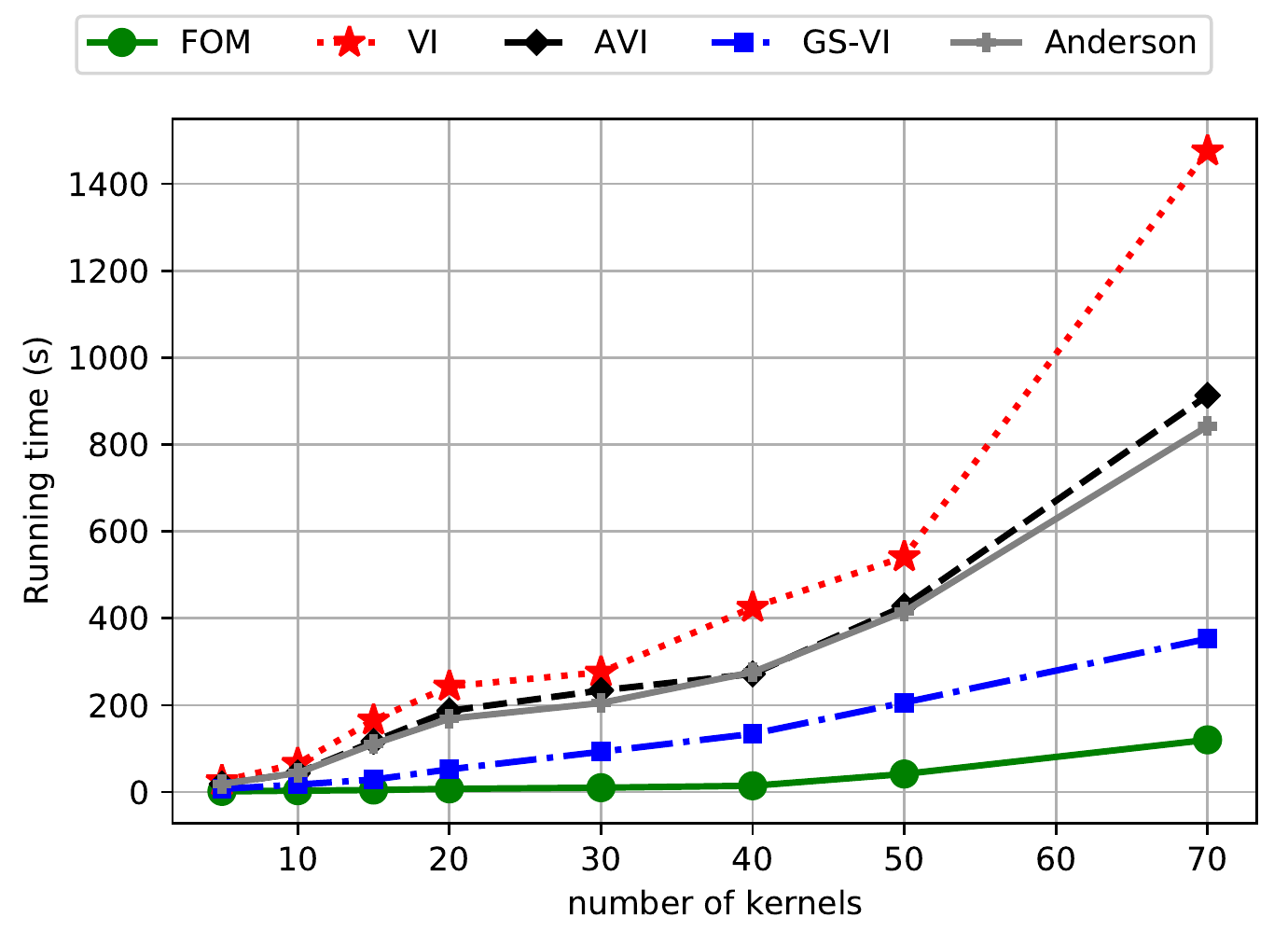}
         \caption{Garnet.}
           \label{fig:kernel_garnet}
  \end{subfigure}
%  \begin{subfigure}{0.24\textwidth}
%\centering
%         \includegraphics[width=1.0\linewidth]{figures/d_2_states.pdf}
%         \caption{**Garnet (high)**.}
%         \label{fig:d_2_states_Garnet}
%  \end{subfigure}
  \caption{Comparison of Alg. \ref{alg:PD-RMDP} with four variants of Value Iteration on three MDP domains (increasing number of kernels).}
  \label{fig:kernel}
\end{figure*}

\begin{figure*}[htp]
  \begin{subfigure}{0.32\textwidth}
\centering
         \includegraphics[width=1.0\linewidth]{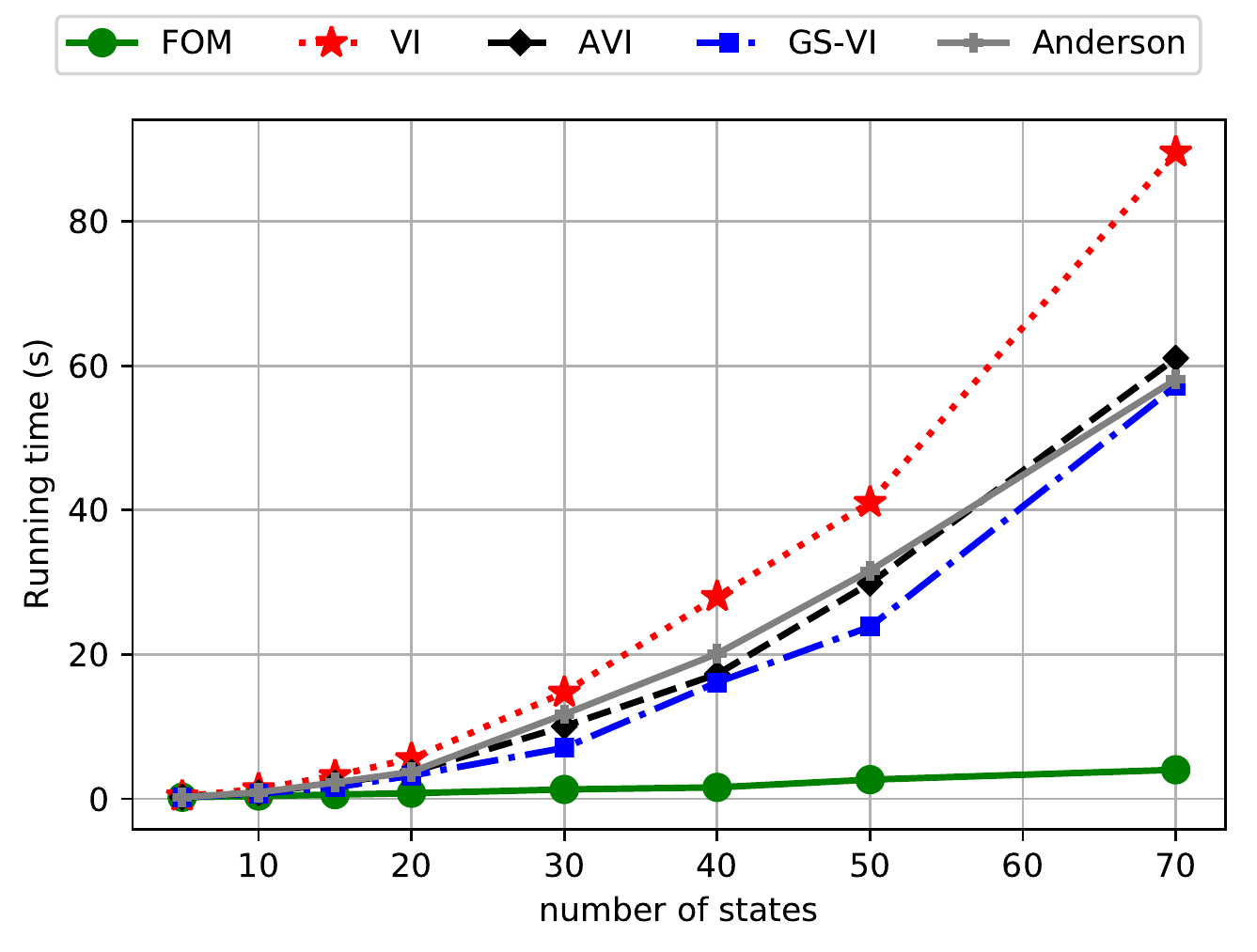}
         \caption{Forest.}
                    \label{fig:state_forest}
  \end{subfigure}
  \begin{subfigure}{0.32\textwidth}
\centering
         \includegraphics[width=1.0\linewidth]{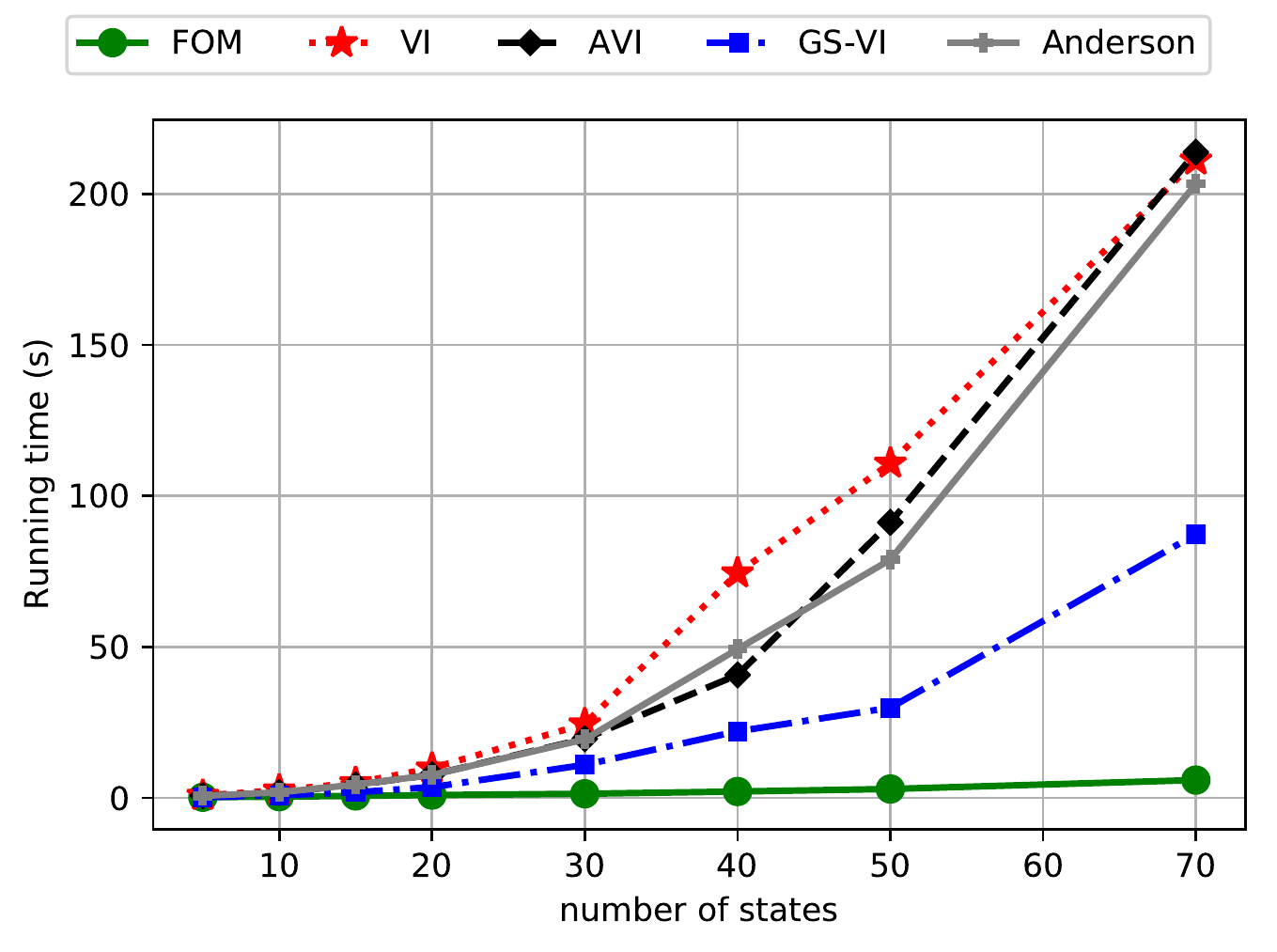}
         \caption{Machine.} 
                 \label{fig:state_machine}
  \end{subfigure}
\begin{subfigure}{0.32\textwidth}
\centering
         \includegraphics[width=1.0\linewidth]{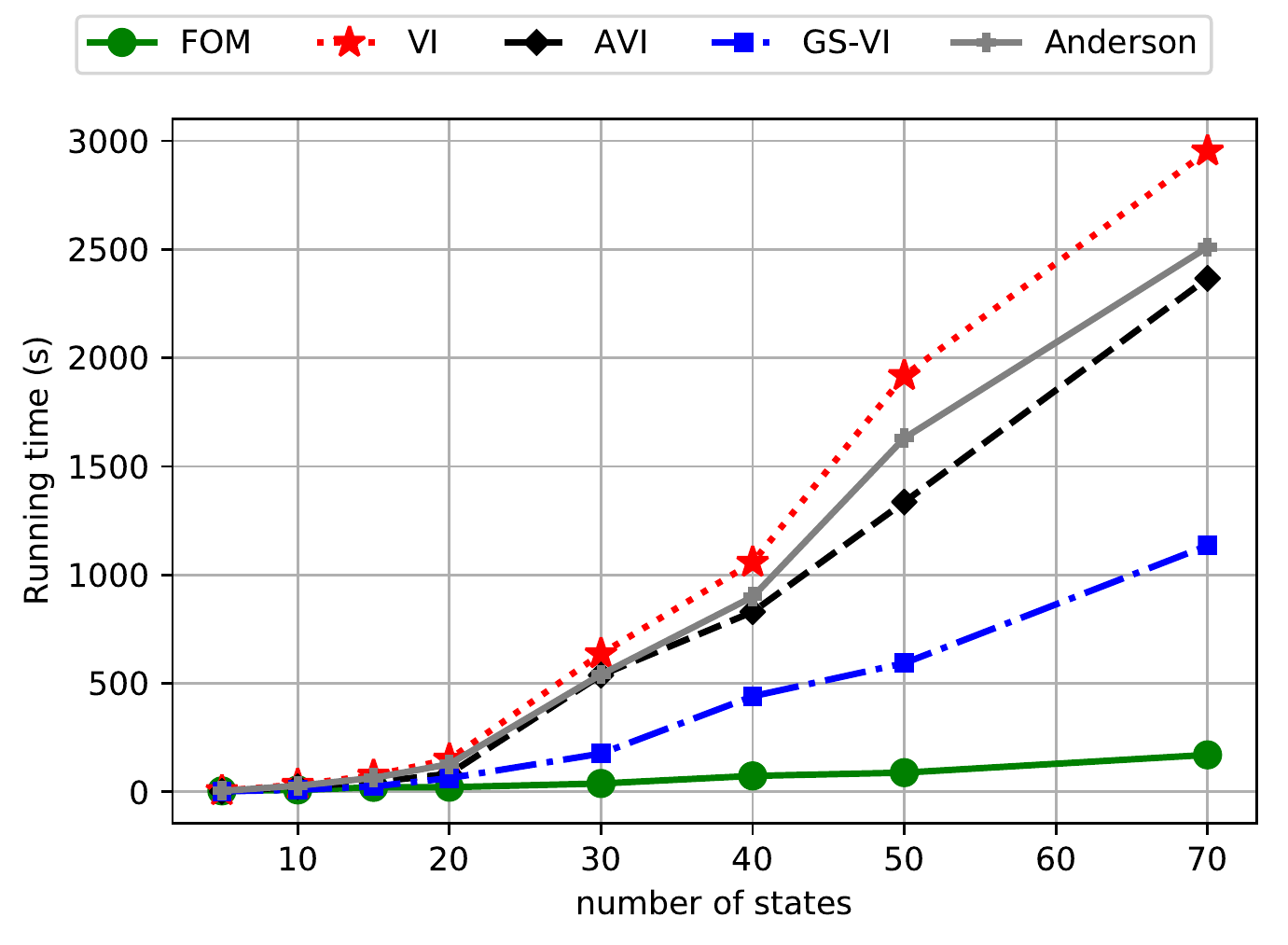}
         \caption{Garnet.}
           \label{fig:state_garnet}
  \end{subfigure}
%  \begin{subfigure}{0.24\textwidth}
%\centering
%         \includegraphics[width=1.0\linewidth]{figures/d_2_states.pdf}
%         \caption{**Garnet (high)**.}
%         \label{fig:d_2_states_Garnet}
%  \end{subfigure}
  \caption{Comparison of Alg. \ref{alg:PD-RMDP} with four variants of Value Iteration on three MDP domains (increasing number of states).}
 \label{fig:state}
\end{figure*}
In this section we compare the empirical performances of our algorithm with state-of-the-art approaches.
We focus on $d=d_{2}$ and we compare the running time of Algorithm \ref{alg:PD-RMDP} to the classical Value Iteration algorithm \ref{alg:VI}, Gauss-Seidel VI (\textit{GS-VI}, \citet{Puterman}), Anderson VI (\textit{Anderson}, \citet{ref-c}), and Accelerated VI (\textit{AVI}, \citet{GGC-AVI}) (see Appendix \ref{app:simu} for more details).

\noindent\textit{Empirical setup.}
We implement our algorithms in Python 3.7.3, using Gurobi 8.1.1 to solve any linear/quadratic optimization program involved. We run our simulations on a laptop with 2.2 GHz Intel Core i7 and 8 GB of RAM.  We test our algorithm on three different sets of instances: a machine replacement problem, a forest management problem and some random (Garnet) instances.  
The discount factor is fixed at $\lambda=0.8$. 
 For each MDP instance, we generate the sampled kernels $\hat{\bm{y}}_{1}, ..., \hat{\bm{y}}_{N}$ by considering $N$ small random (Garnet) perturbations around the ``true'' nominal kernel $\bm{y}^{0}$ (see Appendix \ref{app:simu}).

All figures in this section show the running times of the algorithms before returning an $\epsilon$-optimal policy with $\epsilon=0.1$. We stop Algorithm \ref{alg:PD-RMDP} when \eqref{eq:DG-RMDP} $\leq \epsilon/2$, where
\begin{equation}\label{eq:DG-RMDP}\tag{DG}
 \max_{\mu \in \D} C(\bm{x},\mu) - \min_{\bm{x'} \in \Pi} C(\bm{x'},\mu)
\end{equation}
is the \textit{duality gap} of a pair of policy-density $(\pi,\mu)$.  
Note that \eqref{eq:DG-RMDP} $\leq \epsilon/2 $ is enough to ensure that the policy is an $\epsilon$-optimal policy for \eqref{eq:DR-MDP}. We stop VI and variants when $\| \bm{v}^{\ell} - F(\bm{v}^{\ell}) \|_{\infty} < 2\lambda \epsilon (1-\lambda)^{-1}$, which guarantees that the current policy is $\epsilon$-optimal \cite{Puterman}.
The running times are averaged across 5 instances by changing the seeds for sampling the $N$ kernels around $\bm{y}^{0}$.
 
 \noindent\textit{Initialization and warm-start.}
We initialize all algorithms with $\bm{v}_{0}=\bm{0}$. We evaluate $F(\bm{v})$ using our convex reformulation (see Appendix \ref{app:Belman-update}).
At epoch $\ell$ of \ref{alg:VI} and variants, we warm-start each computation of $F(\bm{v}^{\ell})$ with the optimal solution obtained from the previous epoch $\ell-1$. We present details about the computation of \eqref{eq:DG-RMDP} in Appendix \ref{app:DG}.

\noindent\textit{Structured MDP instances.}
We consider two instances inspired from real-world applications,  a machine replacement problem studied by \citet{Delage}, \citet{Kuhn} and \citet{GGC}, and a forest management example from the Python package \textit{pymdptoolbox} \cite{pymdp} inspired by \citet{possingham1997application}.  In the machine replacement problem, the goal is to design a replacement policy for a line of machines.  The states of the MDP represent age phases of the machine and the actions represent different repair or replacement options.  In the forest management problem,  the forest grows at every period and the goal is to balance the revenue associated with selling cut wood and the risk of wildfire. 
The transition kernels $\bm{\hat{y}}_{1}, ..., \bm{\hat{y}}_{N}$ represent historical data,  obtained from observations from previous years.
In both instances, even though the transition parameters can be estimated from retrospective data sets, one often does not have access to enough data to exactly assess the probability of a machine breaking down when in a given condition, the rate of growth of the forest, or the risk of wildfire.  Additionally, the historical data may contain errors; this warrants the use of a robust model for finding good, stable machine replacement and forest management policies. We present details on these instances in Appendix \ref{app:machine} and Appendix \ref{app:forest}.
 
\noindent\textit{Random MDP instances.}
We also test our algorithm on random, denser MDP instances. We use the Generalized Average Reward Non-stationary Environment Test-bench, or in short, Garnet MDPs \cite{garnet,bhatnagar2007naturalgradient}. Garnet MDPs are a class of abstract but representative finite MDPs that are easy to build and for which we can control the connectivity of the underlying Markov chain with a \textit{branching} factor, $n_{b}$, which represents the proportion of next states available at every state-action pair $(s,a)$. They are a class of randomly constructed finite MDP’s serving as a test-bench for
RL algorithms \cite{garnet-1,garnet-2,garnet-3}.  
We consider $S=A$, $n_{b}=20 \%$ and random uniform rewards in $[0,10]$.

\noindent\textit{Increasing instance sizes.}
Our experiments evaluate the performance of all algorithms by running them on increasingly-larger instances. Our problems have three size parameters: $S$ and $A$, which affect the MDP size, and $N$, which affects the size of the ambiguity sets. Because the runtimes of the VI algorithms grow quickly in these parameters, we perform our experiments by holding two out of three parameters fixed, while increasing the last one. When we consider an increasing number of kernels (Figure \ref{fig:kernel}), we keep $S = 30$ fixed. When we consider an increasing number of states, we keep $N =30$ fixed. For all instances, $A=30$ for Garnet MDPs, $A=2$ for machine replacement MDPs, and $A=3$ for forest management MDPs.

% The running times of the Value Iteration algorithms quickly becomes very large (e.g.  Figure \ref{fig:kernel_garnet}, Figure \ref{fig:state_garnet}).  For this reason,  

\noindent\textit{Numerical results.}
We present the results of our numerical study in Figure \ref{fig:kernel} and Figure \ref{fig:state}.  For very small instances (e.g.  $S=5$ states,  $A=2$ actions, $N=30$ observed kernels),  Algorithm \ref{alg:PD-RMDP} has similar performance as the other four algorithms.  When the number of states or the number of kernels increases,  the average convergence times of our algorithm moderately increase,  e.g.  from  $1.6$ seconds for $N=5, S,A=30$ to $120.2$ seconds for $N=70,S,A=30$ (Figure \ref{fig:kernel_garnet}).  However, Algorithm \ref{alg:PD-RMDP} scales significantly better than the other methods based on IPM, and as the instance sizes increases it outperforms all other methods.  We also see that for Garnet instances, the convergences of all the algorithms are slower, since the MDP instances are denser than for more structured examples and there are more actions.
As expected from our theoretical results in the previous section, the running time of Algorithm \ref{alg:PD-RMDP}  grows linearly with $N$.  Perhaps more surprisingly,  the empirical running times of the other algorithms also seem to grow (almost) linearly with $N$. This may be due to the solver (Gurobi 8.1.1) exploiting the particular problem structure of the robust Bellman update (see Appendix \ref{app:Belman-update}).

%%% Local Variables:
%%% mode: plain-tex
%%% TeX-master: "AAAI"
%%% End:

%\section{Conclusion and Discussion}
%We present a first-order framework for Distributionally Robust MDPs with Wasserstein distance balls, based on adapting the first-order framework for Robust MDP introduced in \citet{grand2020scalable}. Our algorithms rely on developping novel efficient algorithms for the proximal updates over Wasserstein balls. For the case of Wasserstein distances based on metric $d_{1}, d_{2}$ and $d_{\infty}$, our algorithm has a theoretical convergence which improves upon Value Iteration, in terms of the dependence on the number of kernels in the support of the nominal distribution, states and actions. Our numerical experiments highlight the significant speedups compared to state-of-the-art algorithms, both on random and structured MDP instances.

\bibliography{FOM_RMDP}
\bibliographystyle{icml2021}
\appendix
\section{Convex reformulation for Bellman update}\label{app:Belman-update}
We show here how to reformulate \eqref{eq:Bellman-DR-MDP-B} into a convex program, for $\B_{s} = \B_{p,s}$ (the reformulation for $\B_{s} = \B_{\infty,s}$ follows directly).
 At every epoch of Value Iteration \ref{alg:VI}, we compute $F(\bm{v})$ for the current value vector $\bm{v} \in \R^{S}$, where
\begin{equation*}
F(\bm{v})_{s} = \min_{\bm{x}_{s} \in \Delta(A)} \max_{\bm{y}_{s} \in \B_{s}} \sum_{a=1}^{A} x_{sa} \left( c_{sa} + \lambda \cdot \bm{y}_{sa}^{\top}\bm{v} \right), \forall \; s \in \X.
\end{equation*} 
From convex duality we have, for any $s \in \X$,
\begin{align}
F(\bm{v})_{s} & = \min_{\bm{x}_{s} \in \Delta(A)} \max_{\bm{y}_{s} \in \B_{s}}  \sum_{a=1}^{A} x_{sa} \left( c_{sa} + \lambda \cdot \bm{y}_{sa}^{\top}\bm{v} \right)   \nonumber \\
& = \max_{\bm{y}_{s} \in \B_{s}}  \min_{\bm{x}_{s} \in \Delta(A)}  \sum_{a=1}^{A} x_{sa} \left( c_{sa} + \lambda \cdot \bm{y}_{sa}^{\top}\bm{v} \right). \label{eq:T_v_max_min}
\end{align}
For $\bm{y} \in \B_{s}$, we can reformulate the inner minimization as
%\begin{align*}
%\min_{\bm{x}_{s} \in \Delta(A)}  \sum_{a=1}^{A} x_{sa} \left( c_{sa} + \lambda \cdot \bm{y}_{sa}^{\top}\bm{v} \right) 
%& = \min_{\bm{x}_{s} \geq \bm{0}, \bm{x}_{s}^{\top}\bm{e}=1 }  \sum_{a=1}^{A} x_{sa} \left( c_{sa} + \lambda \cdot \bm{y}_{sa}^{\top}\bm{v} \right) \\
%& = \min_{\bm{x}_{s} \geq \bm{0} } \max_{\mu \in \R}  \sum_{a=1}^{A} x_{sa} \left( c_{sa} + \lambda \cdot \bm{y}_{sa}^{\top}\bm{v} \right) + \mu \left( 1-\sum_{a=1}^{A} x_{sa} \right)  \\
%& = \min_{\bm{x}_{s} \geq \bm{0} } \max_{\mu \in \R }\mu +  \sum_{a=1}^{A} x_{sa} \left( c_{sa} + \lambda \cdot \bm{y}_{sa}^{\top}\bm{v} -\mu \right) \\
%& = \max_{\mu \in \R} \min_{\bm{x}_{s} \geq \bm{0} } \mu +  \sum_{a=1}^{A} x_{sa} \left( c_{sa} + \lambda \cdot \bm{y}_{sa}^{\top}\bm{v} -\mu\right),
%\end{align*}

\begin{align*}
\max \; & \gamma \\
& \gamma \in \R, \\
& c_{sa} + \lambda \bm{y}_{sa}^{\top}\bm{v} \geq \gamma, \forall \; a \in \A.
\end{align*}
Overall, we have proved that
\begin{equation}\label{eq:T_v_simu}
\begin{aligned}
F(\bm{v})_{s} = \max \; & \gamma \\
& \gamma \in \R, \bm{y} \in \B_{s},\\
& c_{sa} + \lambda \bm{y}_{sa}^{\top}\bm{v} \geq \gamma, \forall \; a \in \A.
\end{aligned}
\end{equation}
Replacing $\B_{s}$ by $\tilde{\B}_{p,s}$ we obtain
\begin{equation}\label{eq:T_v_simu_final}
\begin{aligned}
F(\bm{v})_{s} = \max \; & \gamma \\
& \gamma \in \R, \bm{y}_{1}, ..., \bm{y}_{N} \in \U ,\\
& c_{sa} + \lambda \dfrac{1}{N} \sum_{i=1}^{N} \bm{y}_{i,sa}^{\top}\bm{v} \geq \gamma \forall \; a \in \A, \\
& \dfrac{1}{N} \sum_{i=1}^{N} d(\bm{y}_{i}, \bm{\hat{y}}_{i,s})^{p} \leq \theta^{p}.
\end{aligned}
\end{equation}
Formulation \eqref{eq:T_v_simu_final} is a linear program with linear constraints (for $d=d_{1},d_{\infty}$ and $p=1$), and one additional quadratic constraint (for $d=d_{2}$ and $p=2$). Following \cite{BenTal-Nemirovski}, we can solve \eqref{eq:T_v_simu_final} up to accuracy $\epsilon$ in a number of arithmetic operations in $O\left(N^{3.5}S^{3.5}A^{3.5}\log(1/\epsilon) \right).$ We warm-start each of this optimization problem with the optimal solution found in the previous epoch of \ref{alg:VI}.
\section{Proof of Theorem \ref{th:conv-T}}\label{app:proof-main-th}
We present here the detailed proof for Theorem \ref{th:conv-T}.
We proceed in three steps:
\begin{itemize}
\item We justify the choice of the step-sizes $\sigma, \tau$ as
 \[ \tau = \left( \sqrt{A} \lambda \| \bm{v}^{\ell} \|_{2} \right)^{-1}, \sigma = N \sqrt{A} (\lambda \| \bm{v}^{\ell} \|_{2})^{-1}.\]
\item We prove upper bounds on the duality gap \eqref{eq:duality-gap-in-bellman}.
\item We finally combine these upper bounds to obtain the convergence rate of Theorem \ref{th:conv-T}.
\end{itemize}
\paragraph{Choice of step-sizes.}
We define \[L=  \sup_{ \| \bm{x} \|_{2} \leq 1, \| (\bm{y})_{i} \|_{2} \leq 1}  \sum_{a \in \A} \bm{x}_{sa}\lambda \sum_{i=1}^{N} \dfrac{1}{N} \bm{y}_{i,sa}^{\top}\bm{v}^{\ell}.\]
At epoch $\ell$ we choose step sizes $\sigma, \tau$ such that
\begin{equation}\label{eq:step-size-condition}
\dfrac{1}{ \sqrt{\sigma \tau}} = L.
\end{equation}
From \citet{ChambollePock16}, this is enough to ensure that $\bar{\bm{x}}^{\ell}_{s}, (\bar{\bm{y}}^{\ell}_{i,s})_{i}$ are $O(1/\ell^{2})$-optimal in computing $F(\bm{v}^{\ell})$, where $\bar{\bm{x}}^{\ell}_{s}, (\bar{\bm{y}}^{\ell}_{i,s})_{i}$ are the weighted averages for the iterates \[(\bm{x}_{\tau_{\ell}+1},(\bm{y}_{\tau_{\ell}+1,i})_{i}), ..., (\bm{x}_{\tau_{\ell}+\ell^{2}},(\bm{y}_{\tau_{\ell}+\ell^{2},i})_{i}),\] with weights $\tau_{\ell}+1,...,\tau_{\ell} + \ell^{2}$.
Now note that, by using Cauchy-Schwarz twice, we have
\begin{equation}\label{eq:L-sqrt-N}
L = \dfrac{\lambda }{\sqrt{N}} \| \bm{v}^{\ell} \|_{2}.
\end{equation}
Note that we could simply choose 
$\sigma = \tau = \sqrt{N} \left(  \lambda \| \bm{v}^{\ell} \|_{2} \right)^{-1}. $
However, since our convergence rate will involve the term $\Theta_{X} / \tau + \Theta_{Y}/\sigma$, we try to equalize these two terms. 
Under the condition \eqref{eq:step-size-condition}, the best choice of step sizes is therefore $ \tau = \left( \sqrt{\Theta_{X} / \Theta_{Y}} \right) L^{-1}$.  Recall that $\Theta_{X}, \Theta_{Y}$ are the maximum of the respective Bregman divergences (squared norm two) onto $\Delta(A)$ and $\tilde{B}_{p,s}$. Therefore, $\Theta_{X} = O(1), \Theta_{Y} = O( NA)$.This leads to 
\[ \tau = \left( \sqrt{A} \lambda \| \bm{v}^{\ell} \|_{2} \right)^{-1}, \sigma = N \sqrt{A} (\lambda \| \bm{v}^{\ell} \|_{2})^{-1}. \]
Note that we are essentially adjusting the step sizes, taking into account the difference of dimensions between $\Delta(A)$, the decision space of the min-player, and $\tilde{B}_{p,s} \subset \R^{N \times A \times S}$, the decision space of the max-player.
\paragraph{Upper bounds on duality gap \eqref{eq:duality-gap-in-bellman}}
Note that Theorem 3.1 in \citet{grand2020scalable} only gives an upper bound on \eqref{eq:duality-gap-in-bellman} when $N=1$, which reduces to the case of robust MDP. However, note that we can extend this result to distributionally robust MDPs by considering that Algorithm \ref{alg:PD-RMDP} is running $N$ instances of the same algorithm for robust MDPs, one instance per kernel $\bm{y}_{i}$. Here it is crucial to reckon that:
\begin{itemize}
\item This scales the constants $R_{Y}$ (maximum of $\| \cdot \|_{Y}$ on $Y$) and $\Theta_{Y}$ (maximum of $D_{Y}$ on $Y \times Y$). This is because for distributionally robust MDPs with nominal distribution supported on $N$ kernels, $Y$ is now contained in $\left( \Delta(S) \right)^{N \times A}$, compared to $Y$ contained in $\left( \Delta(S) \right)^{A}$ for robust MDPs; here recall that we denote by $Y$ the decision space of the max-player. 
\item This leaves unchanged the convergence rate of Algorithm \ref{alg:PD-RMDP} \textit{in terms of number of PD iterations} $T$, as this convergence rate only depends (in terms of transition kernels) of the \textit{expected value} $\bm{y}_{s}=\dfrac{1}{N} \sum_{i=1}^{N} \bm{y}_{i,s}$.
\end{itemize}
Therefore, after $T$ PD iterations of Algorithm \ref{alg:PD-RMDP}, the duality gap \eqref{eq:duality-gap-in-bellman} is upper bounded by 
\[O \left(  R_{X}R_{Y} \left( \dfrac{\Theta_{X}}{\tau} + \dfrac{\Theta_{Y}}{\sigma} \right) \dfrac{\sqrt{S}}{\sqrt{N}} \left( \dfrac{\lambda^{T^{1/3}}}{T^{1/3}} + \dfrac{1}{T^{2/3}} \right) \right).\]
Note the additional $1/\sqrt{N}$, compared to Theorem 3.1 from \citet{grand2020scalable}; this comes from the equality \eqref{eq:L-sqrt-N}.
Let us now simplify this upper bound.
Note that 
\[ \dfrac{\lambda^{T^{1/3}}}{T^{1/3}} + \dfrac{1}{T^{2/3}}  = O \left( \dfrac{1}{T^{2/3}} \right),\]
because of the exponential decay of the term $\lambda^{T^{1/3}}$.
Combining the two previous simplifications, we obtain that after $T$ PD iterations, the duality gap \eqref{eq:duality-gap-in-bellman} is upper bounded by
\[ O \left(  R_{X}R_{Y} \left( \dfrac{\Theta_{X}}{\tau} + \dfrac{\Theta_{Y}}{\sigma} \right) \dfrac{\sqrt{S}}{\sqrt{N}}\dfrac{1}{T^{2/3}} \right). \]%We have
%
%\[ \tau = \left( \sqrt{A} \lambda \| \bm{v}^{\ell} \|_{2} \right)^{-1}, \sigma = N \sqrt{A} (\lambda \| \bm{v}^{\ell} \|_{2})^{-1}. \]  Note that $\bm{v}^{\ell}$ is the cost associated with playing policies-kernels \[\bar{\bm{x}}^{\ell}_{s}, (\bar{\bm{y}}^{\ell}_{s,i})_{i}, ...,\bar{\bm{x}}^{0}_{s}, (\bar{\bm{y}}^{0}_{s,i})_{i}\]
%for $\ell$ periods. Therefore, using the equivalence of $\| \cdot \|_{2}$ and $\| \cdot \|_{\infty}$ it is straightforward that
%\[ \dfrac{1}{\sigma} \leq \dfrac{\lambda \left( \max_{s,a} c_{s,a} \right) \sqrt{S}}{(1-\lambda)\sqrt{N}}.\]

\section{Proof Proposition \ref{prop:finite-type}}\label{app:computing-FOM}
In this section we focus on solving \eqref{eq:prox-update-min-player-p}, dropping the index $s \in \X$, with the understanding that $\bm{h} = \bm{h}_{s} \in \R^{A \times S}, \bm{\hat{y}}_{i} = \bm{\hat{y}}_{i,s} \in \U$.
\paragraph{Proof for $d=d_{2}, p=2$.}
The proximal update becomes
\begin{align*}
\min & \sum_{i=1}^{N} \langle \bm{y}_{i} , \bm{h} \rangle + \dfrac{1}{2 \sigma} \|\bm{y}_{i} - \bm{y'}_{i} \|^{2}_{2} \\
& \bm{y}_{1}, ..., \bm{y}_{N} \in  \left( \Delta(S) \right)^{A}, \\
&  \dfrac{1}{N} \sum_{i=1}^{N} \| \bm{y}_{i}- \bm{\hat{y}}_{i}\|_{2}^{2} \leq \theta^{2}.
\end{align*}
If we dualize the second constraint with a Lagrange multiplier $\gamma$, we end up with computing $N A$ Euclidean projections onto the simplex $\Delta(S)$, because the argmin of
\[ \bm{y} \in \U \mapsto \langle \bm{y}, \bm{h} \rangle + \dfrac{1}{2 \sigma} \| \bm{y} - \bm{y}' \|_{2}^{2} + \dfrac{\gamma}{2} \| \bm{y} - \bm{\hat{y}}_{i} \|_{2}^{2} \]
is the same as the argmin of
\[ \bm{y} \in \U \mapsto \dfrac{1}{2} \| \bm{y} - \dfrac{\sigma}{1+\sigma \gamma} \left( \dfrac{1}{\sigma} \bm{y'} + \gamma \bm{\hat{y}}_{i} - \bm{h}\right) \|_{2}^{2}.\]
We therefore compute $NA$ 
Euclidean projections onto the simplex of size $S$, which can be performed in $O \left( NA S \log(S) \right)$ arithmetic operations. We then need to binary search over the Lagrange multiplier $\gamma$, resulting in a complexity $O \left(NA S \log(S) \log(\epsilon^{-1} \right)$.

\paragraph{Proof for $d=d_{1}, p=1$.}
The proximal update becomes
\begin{align*}
\min & \sum_{i=1}^{N} \langle \bm{y}_{i} , \bm{h} \rangle + \dfrac{1}{2 \sigma} \|\bm{y}_{i} - \bm{y'}_{i} \|^{2}_{2} \\
& \bm{y}_{1}, ..., \bm{y}_{N} \in  \left( \Delta(S) \right)^{A}, \\
&  \dfrac{1}{N} \sum_{i=1}^{N} \| \bm{y}_{i}- \bm{\hat{y}}_{i}\|_{1} \leq \theta.
\end{align*}
We introduce a Lagrange multiplier $\gamma \geq 0$ for the second constraint: we now solve
\begin{align*}
\max_{\gamma \geq 0} - \gamma \theta \\
+ \min & \sum_{i=1}^{N} \langle \bm{y}_{i} , \bm{h} \rangle + \dfrac{1}{2 \sigma} \|\bm{y}_{i} - \bm{y'}_{i} \|^{2}_{2} + \gamma \| \bm{y}_{i}- \bm{\hat{y}}_{i}\|_{1}  \\
& \bm{y}_{1}, ..., \bm{y}_{N} \in  \left( \Delta(S) \right)^{A}.
\end{align*}
We then introduce Lagrange multipliers $\left( \alpha_{i,a} \right)_{i,a}$ for each constraint $\sum_{s'=1}^{S} y_{i,a,s'} =1$ for each $i=1, ..., N$ and $a \in \A$:
\begin{align*}
& \max_{\gamma \geq 0} \max_{(\alpha_{i,a})_{i,a} \in \R^{N \times A}} - \sum_{i,a} \alpha_{i,a} - \gamma \theta \\
& + \sum_{i=1}^{N} \sum_{a=1}^{A} \sum_{s'=1}^{S} \min_{y_{i,a,s'} \geq 0}  (h_{i,a,s'} + \alpha_{i,a}) y_{i,a,s'} \\
& + \dfrac{1}{2 \sigma} (y_{i,a,s'} - y_{i,a,s'}' )^{2} + \gamma | y_{i,a,s'} - \hat{y}_{i,a,s'} |.
\end{align*}
\paragraph{Solving the inner minimization.}
Let us drop the index $(i,a,s')$ and explain how to compute a closed-form solution to the inner univariate minimization:
\[ \min_{y \geq 0}  \; (h+ \alpha) y + \dfrac{1}{2 \sigma} (y - y' )^{2} + \gamma | y - \hat{y} |. \]
We can distinguish three regions.
\begin{enumerate}
\item $y >\hat{y}$. The first-order conditions yield
\[ (h+\alpha) + (1/\sigma)  (y - y') + \gamma =0, \]
which implies $y =  y'- \sigma(\gamma+ h+\alpha)$. This is valid as long as $y'- \sigma(\gamma+ h+\alpha) > \hat{y}$. Note that $y'- \sigma( \gamma + h +\alpha) > \hat{y}$ implies $y'- \sigma( \gamma + h +\alpha) \geq 0$, since $\hat{y} \geq 0$.
\item $y < \hat{y}$. The first-order conditions yield $y=y'-\sigma(-\gamma + h + \alpha)$, which is valid as long as 
$y' - \sigma(-\gamma+ h+  \alpha) < \hat{y}$ and $y' - \sigma(-\gamma + h + \alpha) \geq 0$.
\end{enumerate}
Overall, we have
\begin{equation}\label{eq:y-shrinkage}
y = \begin{cases} y'- \sigma(\gamma + h + \alpha) &\mbox{if } \dfrac{1}{\sigma} (y'-\hat{y}) - h -\alpha > \gamma, \\
\hat{y} & \mbox{if } | \dfrac{1}{\sigma} (y'-\hat{y}) - h -\alpha  | \leq \gamma, \\
(y'-\sigma(-\gamma + h + \alpha))^{+} &\mbox{if } \dfrac{1}{\sigma} (y'-\hat{y}) - h -\alpha < -\gamma. \end{cases}
\end{equation}
 Note that this is essentially the shrinkage-thresholding operator, up to the last case and the $x \mapsto x^{+}$ function (which stems from the non-negativity constraint).
 
% If we account for the index $(i,a,s')$ we have
% \begin{equation*}
%y_{i,a,s'} = \begin{cases} y'_{i,a,s'}-\gamma - c_{i,a,s'} - \alpha_{i,a} &\mbox{if } y'_{i,a,s'}-\hat{y}_{i,a,s'} - c_{i,a,s'} -\alpha_{i,a} > \gamma, \\
%\hat{y}_{i,a,s'} & \mbox{if } | y'_{i,a,s'}-\hat{y}_{i,a,s'} - c_{i,a,s'} -\alpha_{i,a}  | \leq \gamma, \\
%(y_{i,a,s'}'+\gamma - c_{i,a,s'} - \alpha_{i,a})^{+} &\mbox{if } y'_{i,a,s'}-\hat{y}_{i,a,s'} - c_{i,a,s'} -\alpha_{i,a} < -\gamma. \end{cases}
%\end{equation*}
\paragraph{Solving the maximization over $\alpha$.} For a fixed Lagrange multiplier $\gamma$, our goal is now to solve 
\begin{equation}\label{eq:max_alpha}
\max_{\alpha \in \R} - \alpha + \sum_{s'=1}^{S} (h_{s'}+ \alpha)y_{s'} + \dfrac{1}{2\sigma} (y_{s'} - y'_{s'}) + \gamma | y_{s'} - \hat{y}_{s'} |,
\end{equation}
where $y$ follows \eqref{eq:y-shrinkage}.
Let us rewrite  \eqref{eq:y-shrinkage} with the index $s'$ and split the thresholding at zero into two cases:
%\begin{equation*}
%y_{s'} = \begin{cases} y'_{s'}-\gamma - h_{s'} - \alpha &\mbox{if } y'_{s'}-\hat{y}_{s'} - h_{s'} -\alpha > \gamma, \\
%\hat{y}_{s'} & \mbox{if } | y'_{s'}-\hat{y}_{s'} - h_{s'} -\alpha  | \leq \gamma, \\
%y'_{s'}+\gamma - h_{s'} - \alpha &\mbox{if } y'_{s'}-\hat{y}_{s'} - h_{s'} -\alpha < -\gamma  \\
%0 &\mbox{if } y'_{s'} - h_{s'} -\alpha < -\gamma. \end{cases}
%\end{equation*}
\begin{equation*}
y = \begin{cases} y'- \sigma(\gamma + h + \alpha) &\mbox{if } (1/\sigma) (y'-\hat{y}) - h -\alpha > \gamma, \\
\hat{y} & \mbox{if } | (1/\sigma) (y'-\hat{y}) - h -\alpha  | \leq \gamma, \\
(y'+\gamma - h - \alpha)^{+} &\mbox{if } (1/\sigma) (y'-\hat{y}) - h -\alpha < -\gamma,\\
0 &\mbox{if }(1/\sigma) y'_{s'} - h_{s'} -\alpha < -\gamma. \end{cases}
\end{equation*}
For each $s' \in \X$ there are three breakpoints where the behavior of $y_{s'}$ changes with respect to the choice of $\alpha$:
\begin{enumerate}
\item $(1/\sigma) y'_{s'} - h_{s'} -\alpha = -\gamma$: $y_{s'}$ becomes nonzero at a rate of $-\sigma \alpha$,
\item $(1/\sigma) (y'-\hat{y}) - h -\alpha =-\gamma$: $y_{s'}$ becomes constant at $\hat y_{s'}$,
\item $y (1/\sigma) (y'-\hat{y}) - h -\alpha = \gamma$: $y_{s'}$ grows above $\hat y_{s'}$ at a rate $- \sigma \alpha$.
\end{enumerate}
% three regions of $\alpha$ where the expression of $y_{s'}$ is distinct: 
% \begin{align}\label{eq:regions-alpha}
% I_{1,s'} & = [y'_{s'} - \hat{y}_{s'} - h_{s'} + \gamma, + \infty), \\
% I_{2,s'} & =  [ y'_{s'} - \hat{y}_{s'} - h_{s'} - \gamma, y'_{s'} - \hat{y}_{s'} - h_{s'}+ \gamma], \\
% I_{3,s'} & = (-\infty, y'_{s'} - \hat{y}_{s'} - h_{s'} - \gamma].
% \end{align}
This yields the following algorithm.
\begin{enumerate}
% \item For each 
% \begin{align}\label{eq:breakpoints}
% \{y'_{s'} - \hat{y}_{s'} - h_{s'} - \gamma | s' \in \} \bigcup \{
% y'_{s'} - \hat{y}_{s'} - h_{s'}+ \gamma | s' \in \X \}.
% \end{align}
\item We sort the breakpoints in decreasing order of $\alpha$, which takes time $O(S \log(S))$.
\item At the first breakpoint, $y_{s'} = 0$ for all $s'$.
\item We keep a counter \texttt{num\_active} denoting how many variables change with $\alpha$ at the current breakpoint, initialized at zero.
\item We keep a counter \texttt{sum} denoting the value of $\sum_{s'}y_{s'}$ if we had set $\alpha$ equal to the current breakpoint, initialized at zero.
\item We then iterate through the breakpoints (in decreasing order). Let $\alpha_1,\alpha_2$ be the previous and current breakpoints. At every breakpoint:
  \begin{enumerate}
  \item set $\texttt{sum} += \sigma \texttt{num\_active}\cdot (\alpha_2 - \alpha_1)$.
  \item if $\texttt{sum} > 1$ then stop and go to \ref{bp search set alpha}.
  \item else, we update \texttt{num\_active} based on whether the current variable starts or stops changing at $\alpha_2$, and go to the next breakpoint.
  \end{enumerate}
  \item From the mean value theorem, an optimal $\alpha^{*}$ belongs to the interval $[\alpha_1,\alpha_2]$.
    We find it by setting $\alpha = \alpha_2 - (\texttt{sum} - 1) / (\sigma \texttt{num\_active})$.
    \label{bp search set alpha}
  %By definition, we also know the expression for each $y_{s'}$ for $\alpha \in I^{*}$, since we have ranked the breakpoints: there exists $J_{1}, J_{2}, J_{3}, J_{4}$ a partition of $\X$ such that
% \begin{align*}
% & \sum_{s' \in J_{1}} y'_{s'} - \gamma - h_{s'} - \alpha^{*} + \sum_{s' \in J_{2}} \hat{y}_{s'}  \\ 
% &+ \sum_{s' \in J_{3}} y'_{s'} + \gamma - h_{s'} - \alpha^{*} + \sum_{s' \in J_{4}} 0 = 1. 
% \end{align*}
% We then use the above equation to compute $\alpha^{*}$.
% We can therefore compute $\alpha^{*}$ in $O(Slog(S))$ arithmetic operations. Note that there is no accuracy here - we can compute $\alpha^{*}$ exactly.
\end{enumerate}
There are $NA$ Lagrange multipliers $(\alpha_{ia})_{i,a}$, and we can compute each of them in $O(S\log(S))$, given a Lagrange multiplier $\gamma$. We still need to use bisection to compute $\gamma^{*}$. Overall we end up  with a complexity of $O(NA^{2}S^{3} \log(\epsilon^{-1})\epsilon^{-1}).$
\begin{remark} In the context of robust MDP (i.e. N=1), note that \citet{Ho} gives an algorithm with complexity $O(S^{2}A\log(S^{2}A))$ to compute \eqref{eq:Bellman-DR-MDP-B} with $d=d_{1},p=1$. It remains unclear to us if this algorithm extends to the case $N \geq 2$ and its complexity in this case.
\end{remark}
%\todo[inline]{Is it clear that the objective is concave in $\alpha$?}
%Note that within some of the regions, the objective is quadratic in $\alpha$. For instance, assume that for $s' \in \X$ we have $y'_{s'} - \hat{y}_{s'} - c_{s'} - \lambda \geq \alpha$; then $y_{s'} = y_{s'}' - \lambda - c_{s'} - \alpha$, and we have
%\begin{align*}
% (c_{s'}+ \alpha)y_{s'} + \dfrac{1}{2\sigma} (y_{s'} - y'_{s'}) + \lambda | y_{s'} - \hat{y}_{s'} | & = (c_{s'} + \alpha)  (y_{s'}' - \lambda - c_{s'} - \alpha) + \dfrac{1}{2} \left( \lambda + c_{s'} + \alpha \right)^{2} + \lambda | y'_{s'} - \hat{y}_{s'} - \lambda - c_{s'} - \alpha | \\
% & = (c_{s'} + \alpha)  (y_{s'}' - \lambda - c_{s'} - \alpha) + \dfrac{1}{2} \left( \lambda + c_{s'} + \alpha \right)^{2} + \lambda \left(  y'_{s'} - \hat{y}_{s'} - \lambda - c_{s'} - \alpha \right) \\
% & \equiv (y'_{s'} - \lambda - 2 c_{s'}) \alpha - \alpha^{2} + \dfrac{1}{2} \alpha^{2} +  (\lambda + c_{s'}) \alpha +  \lambda \alpha \\ 
% & \equiv - \dfrac{1}{2} \alpha^{2} + (y'_{s'} + \lambda - c_{s'}) \alpha.
%\end{align*}
%This is indeed quadratic in $\alpha$.

\paragraph{Proof for $d=d_{\infty},p=1$.}
The FOM update becomes
\begin{align*}
\min & \;  \dfrac{1}{N} \sum_{i=1}^{N} \langle \bm{y}_{i} , \bm{h} \rangle + \dfrac{1}{2\sigma} \| \bm{y}_{i} - \bm{y'}_{i} \|_{2}^{2} \\
& \bm{y}_{1}, ..., \bm{y}_{N} \in \U, \\
&  \dfrac{1}{N} \sum_{i=1}^{N} \|\bm{y}_{i} - \bm{\hat{y}}_{i} \|_{\infty} \leq \theta.
\end{align*}
We introduce a Lagrange multiplier $\gamma \in \R$ for the binding constraint, and our goal is now to solve
\begin{align*}
\min & \; \sum_{i=1}^{N} \langle \bm{y}_{i} , \bm{h} \rangle + \dfrac{1}{2\sigma} \| \bm{y}_{i} - \bm{y'}_{i} \|_{2}^{2} + \gamma \cdot \sum_{i=1}^{N} \|\bm{y}_{i} - \bm{\hat{y}}_{i} \|_{\infty}  \\
& \bm{y}_{1}, ..., \bm{y}_{N} \in \U.
\end{align*}
Note that this problem decomposes across $i=1, ..., N$, so that we can solve independently, for each $i$,
\begin{equation}\label{eq:prox-loc-infty-1}
\begin{aligned}
\min & \; \langle \bm{y} , \bm{h} \rangle + \dfrac{1}{2 \sigma} \|\bm{y} - \bm{y'}_{i} \|_{2}^{2} + \gamma \|\bm{y} - \bm{\hat{y}}_{i} \|_{\infty}  \\
& \bm{y} \in \U.
\end{aligned}
\end{equation}
To solve \eqref{eq:prox-loc-infty-1}, we can use bisection to find a feasible $\alpha$ such that $ \gamma \| \bm{y} - \bm{\hat{y}}_{i} \|_{\infty} \leq \alpha$.
This leads to solve
\begin{align*}
\min & \; \langle \bm{y} , \bm{h} \rangle + \dfrac{1}{2 \sigma} \|\bm{y} - \bm{y'}_{i} \|_{2}^{2}  \\
& \bm{y} \in (\Delta(S))^{A}, \\
& \gamma \| \bm{y}_{a} - \bm{\hat{y}}_{i,a} \|_{\infty} \leq  \alpha, \forall \; a \in \A.
\end{align*}
Note that this problem decomposes across each action $a$, so that we only have to solve $A$ problems of the form
\begin{align*}
\min & \; \langle \bm{y}_{i,a} , \bm{h}_{ia} \rangle + \dfrac{1}{2 \sigma} \| \bm{y}_{i,a} - \bm{y'}_{i,a} \|_{2}^{2}  \\
& \bm{y}_{i,a} \in \Delta(S), \\
& \gamma \| \bm{y}_{i,a}- \bm{\hat{y}}_{i,a} \|_{\infty} \leq \alpha.
\end{align*}
This brings down to solving the problem of Euclidean projection onto the simplex $\Delta(S)$ with box constraints, which can be done in $O(S \log(S) \log(\epsilon^{-1})$ (by relaxing the constraint $\bm{y}_{i,a}^{\top}\bm{e}=1$).
Then the overall complexity to compute an $\epsilon$-approximation of the proximal update is in $O\left( N A S \log(S) \log^{3} \left(\epsilon^{-1})\right) \right).$
\section{Complexity results for type-$\infty$ Wasserstein ball}\label{app:p-infinity}
\paragraph{Background on type-$\infty$ Wasserstein distance}
\citet{xie2020tractable}, \citet{bertsimas2019two}, \citet{bertsimas2018data} consider ambiguity sets based on type-$\infty$ Wasserstein distance with application to two-state distributionally robust optimization. Recent work suggests that distributionally robust optimization based on type-$\infty$ distance has some computational advantages compared to  DRO based on type-$p$ Wasserstein distance \cite{xie2020distributionally}.
\paragraph{Optimality of Markovian policy}
Note that \citet{yang2017convex} proves that for type $p$ Wasserstein distance (with $p<+\infty$), an optimal policy can be found Markovian. We prove here that the same holds for Wasserstein distance of $p=+\infty$. Let us define the \textit{value} vector for each state $s$ as 
\[ v_{s} = \min_{\bm{x} \in \Delta(A)} \max_{\mu_{s} \in \D_{s}}  \E_{\pi} \E_{\bm{y} \sim \mu_{s}} [\sum_{t=0}^{+ \infty} \lambda^{t} c_{s_{t}a_{t}} | \; s_{0} = s ],\]
which represents the expected reward-to-go starting from a state $s$. Note that $s \mapsto v_{s}$ is well-defined because of the $s$-rectangularity assumption \cite{Kuhn}. The Bellman equation \eqref{eq:Bellman-DR-MDP} follows from the dynamic programming principle. Now we have that
\[\bm{x} \mapsto  \max_{\mu_{s} \in \D_{s}} \E_{\bm{y}_{s} \sim \mu_{s}} \left[ \sum_{a \in \A} x_{sa} \left( r_{sa} + \lambda \bm{y}_{sa}^{\top}\bm{v}^{*} \right) \; | s_{0}=s \right] \]
is convex (as the pointwise maximum of linear functions), proper (because the costs are bounded), and upper semi-continuous. Hence the minimization problem over $\bm{x} \in \Delta(A)$ is minimizing a closed proper convex function onto the closed convex set $\Delta(A)$. Therefore an optimal solution exists, i.e. there exists an optimal Markovian policy.
\paragraph{Proximal update.}
The proximal update on the max-player becomes
\begin{equation}\label{eq:prox-update-min-player-infty}
\begin{aligned}
\min & \; \sum_{i=1}^{N} \langle \bm{y}_{i} , \bm{h} \rangle + \dfrac{1}{2\sigma} \| \bm{y}_{i} - \bm{y'}_{i} \|_{2}^{2} \\
& \bm{y}_{1}, ..., \bm{y}_{N} \in \U, \\
&d(\bm{y}_{i}, \bm{\hat{y}}_{i}) \leq \theta, \forall \; i=1,...,N.
\end{aligned}
\end{equation}
We note that this problem naturally decomposes along $i=1, ..., N$, so that we only have to solve $N$ subproblems of the form
\begin{equation}\label{eq:prox-update-min-player-infty-1}
\begin{aligned}
\min & \;  \langle \bm{y}, \bm{h} \rangle + \dfrac{1}{2 \sigma}  \| \bm{y}_{i} - \bm{y'}_{i} \|_{2}^{2} \\
& \bm{y} \in \U, \\
&d(\bm{y}, \bm{\hat{y}}_{i}) \leq \theta.
\end{aligned}
\end{equation}
If we introduce a Lagrange multiplier $\gamma$ for the last constraint, we note that we have to solve
\begin{equation}\label{eq:prox-update-min-player-infty-2}
\begin{aligned}
\min & \; \sum_{i=1}^{N} \langle \bm{y}, \bm{h} \rangle + \dfrac{1}{2 \sigma} \| \bm{y}_{i} - \bm{y'}_{i} \|_{2}^{2} + \gamma \cdot d(\bm{y}, \bm{\hat{y}}_{i}) \\
& \bm{y} \in \U.
\end{aligned}
\end{equation}
It is straightforward to use the same methods as for the proximal updates for $p< + \infty$ and $d= d_{1}, d_{2}, d_{\infty}$, which yields the following corollary of Proposition \ref{prop:finite-type}.
 \begin{corollary}\label{cor:infinite-type}
 \begin{enumerate}
 \item Let $d=d_{2}$ and $p=2$.
The proximal update \eqref{eq:prox-update-min-player-infty-2} can be computed in $O \left(NA S \log(S) \log(\epsilon^{-1}) \right)$ arithmetics operations.

\item Let $d=d_{1}$ and $p=1$.
The proximal update \eqref{eq:prox-update-min-player-infty-2} can be computed in $O \left(NA S \log(S) \log(\epsilon^{-1}) \right)$ arithmetics operations.
\item 
Let $d= d_{\infty}$ and $p=1$.
The proximal update \eqref{eq:prox-update-min-player-infty-2} can be computed in $O \left(NA S \log(S) \log^{3}(\epsilon^{-1}) \right)$ arithmetics operations.
 \end{enumerate}
\end{corollary}
The corresponding convergence rates for Algorithm \ref{alg:PD-RMDP} with $p = + \infty$ are given in Theorem \ref{th:conv-rates}.

\section{Computing the duality gap}\label{app:DG}
Remember that the duality gap in \eqref{eq:DR-MDP}  is defined  as
\[ \max_{\mu \in \D} C(\bm{x},\mu) - \min_{\bm{x'} \in \Pi} C(\bm{x'},\mu).\]
Following \citet{yang2017convex}, $\max_{\mu \in \D} C(\bm{x},\mu) $ can be computed by finding the fixed point of the following operator, which is a contraction of factor $\lambda$:
$F^{\bm{x}}(\bm{v})_{s} = \max_{\mu \in \D_{s}} \E_{\bm{y} \sim \mu} \left[\sum_{a=1}^{A} x_{sa} \left( c_{sa} + \lambda \bm{y}^{\top}\bm{v} \right) \right], \forall \; s \in \X.$
Moreover, computing $\min_{\bm{x'} \in \Pi} C(\bm{x'},\mu)$ is equivalent to solving the (nominal) MDP with fixed density $\mu \in \D$. This can be solved by iterating the following contraction:
$ F^{\bm{y}}(\bm{v})_{s} = \min_{\bm{x_{s}} \in \Delta(A)} \E_{\bm{y} \sim \mu} \left[ \sum_{a=1}^{A} x_{sa} \left( c_{sa} + \lambda \bm{y}^{\top}\bm{v} \right) \right], \forall \; s \in \X.$

We present in the next figure the running times to compute \eqref{eq:DG-RMDP} up to $\epsilon=0.25$, using the numerical setup of our numerical experiments for Garnet MDPs. We present our results for $\lambda=0.8$. We notice that computing \eqref{eq:DG-RMDP} quickly becomes  slow. Therefore, in our experiments we focus on computing \eqref{eq:DG-RMDP} for $S,A,N$ smaller than $70$. 
\begin{figure}[H]
     \centering
         \includegraphics[scale=0.4]{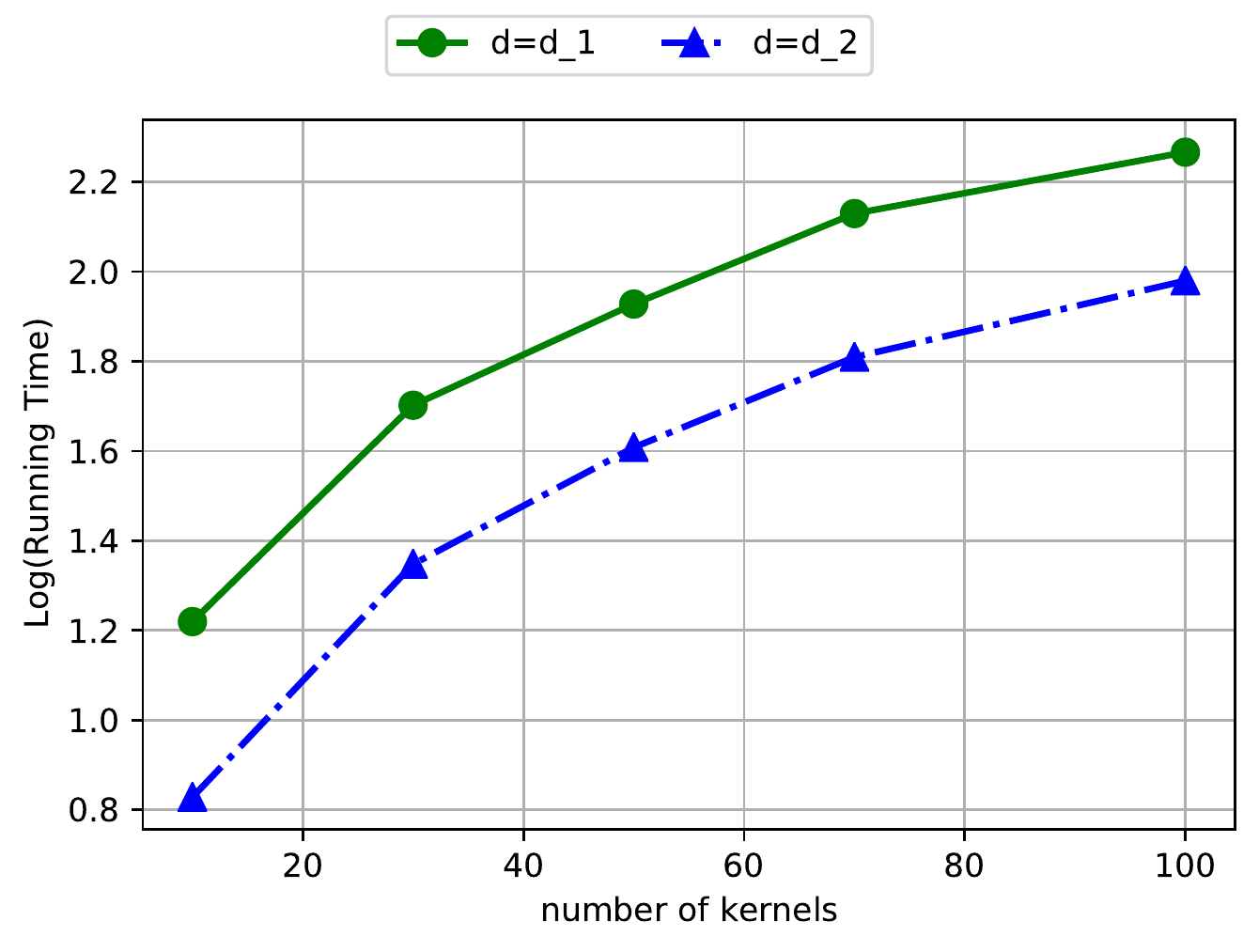}
         \caption{Running times for computing the duality gap \eqref{eq:DG-RMDP}, for increasing number of kernels (while $S,A=10$).}
           \label{fig:DG_kernels}

     \end{figure}
\begin{figure}[H]
         \centering
         \includegraphics[scale=0.4]{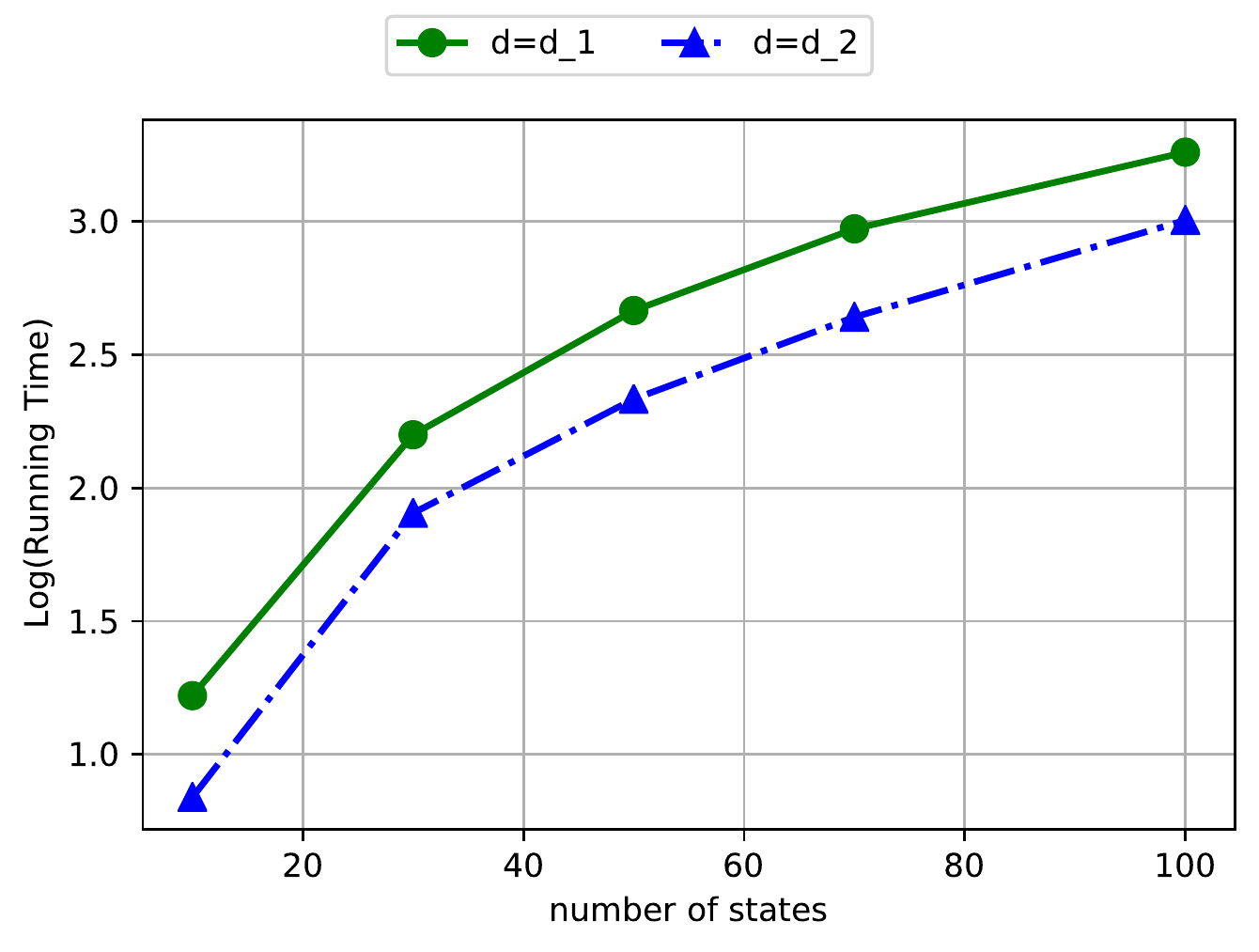}
         \caption{States.}
         \label{fig:DG_states}
        \caption{Running times for computing the duality gap \eqref{eq:DG-RMDP}, for increasing number of states (while $N,A=10$).}
\end{figure}
We also note here that the duality gap is slower to compute for $d=d_{1}$ (where the Bellman update brings down to a large linear program) than for $d=d_{2}$ (where the Bellman update brings down to a convex program with less variables than for $d=d_{1}$ but one additional quadratic constraints). Note that in the case of $d=d_{1}$, $NAS$ additional variables have to be introduced to model the absolute values $| y_{i,a,s'} - \hat{y}_{i,a,s'}|$ for all $i=1, ..., N, a \in \A, s' \in \X$; this is probably what causes the Bellman update with $d=d_{2}$ to be faster, even if it introduces a (single) quadratic constraint.

\section{Details on numerical implementations}\label{app:simu}
\paragraph{Estimating the Bellman operator.} In order to obtain $F(\bm{v})$, we use the reformulation \eqref{eq:T_v_simu_final} and solve it using Gurobi 8.1.1 for Python 3.7.3.  Following \citet{BenTal-Nemirovski}, we can solve \eqref{eq:T_v_simu_final} up to accuracy $\epsilon$ in a number of arithmetic operations in $O\left(S^{3.5}A^{3.5}\log(1/\epsilon) \right).$ 
For VI, AVI, Anderson and GSVI, we warm-start the computation of $F(\bm{v}^{\ell})$ with the previous solution obtained from solving $F(\bm{v}^{\ell-1})$. 

\paragraph{Computing uncertainty sets.} In order to obtain the $N$ transition kernels $\left( \bm{\hat{y}}_{i} \right)_{i=1}^{N}$, we sample some \textit{random} (Garnet) deviations around the true nominal kernel $\bm{y}^{0}$. In particular, we sample $N$ Garnet MDP instances $\bm{y}_{1}, ...., \bm{y}_{N}$ with $n_{b} = 0.05$ (very low level of connectivity), and we consider $\bm{\hat{y}}_{1}, ..., \bm{\hat{y}}_{N}$ as
\[ \bm{\hat{y}}_{i} = 0.95 \bm{y}^{0} + 0.05 \bm{y}_{i}, i = 1,...,N.\]
This way, $\left( \bm{\hat{y}}_{i} \right)_{i=1}^{N}$ represent $N$ kernels, obtained as small (random) errors from the true transition kernel $\bm{y}^{0}$. The nominal kernel for the machine replacement and the forest management instances are given in the next appendices.

For the machine replacement and the forest management problems, we build an uncertainty set of the form \eqref{eq:ambiguity-set-wasserstein} with $\theta = 0.5$.
We choose to present our results for $\theta=0.5$ as they are representative of our results for other choices ($\theta \in \{0.1,0.5,1,2\}$). 
As the Garnet MDPs have denser transitions, we choose $\theta = \sqrt{n_{b}A}$ as the radius for the Wasserstein balls.

\paragraph{Accelerated Value Iteration.} The algorithm AVI  \cite{GGC, akian2020multiply} is a simple variation of VI, inspired from acceleration scheme from convex optimization \cite{nesterov-1983,nesterov-book}.  In particular, for any sequences of scalar $(\alpha_{s})_{s \geq 0}$ and $(\gamma_{s})_{s \geq 0} \in \R^{\N}$, Accelerated Value Iteration (AVI) is defined as
\begin{equation}\label{alg:AVI}\tag{AVI}
\bm{v}_{0},\bm{v}_{1} \in \R^{S}, \begin{cases}
    \bm{h}_{t}=\bm{v}_{t}+\gamma_{t}\cdot \left( \bm{v}_{t}-\bm{v}_{t-1} \right), \\
	\bm{v}_{t+1} \gets \bm{h}_{t}-\alpha_{t} \left( \bm{h}_{t}- F \left( \bm{h}_{t} \right) \right), \end{cases} \forall \; t \geq 1.
\end{equation}
Following \cite{GGC}, we choose step sizes as\begin{equation*}
\alpha_{s}  = \alpha = 1/(1+\lambda),
 \gamma_{s} = \gamma = \left(1-\sqrt{1- \lambda^{2}}\right)/\lambda, \forall s \; \geq 1.
 \end{equation*}

\paragraph{Gauss-Seidel Value Iteration.}  Gauss-Seidel Value Iteration (GS-VI) is a popular asynchronous variant of \ref{alg:VI} \citep{Puterman}, where 
$
v^{t+1}_{s} = \max_{a \in \A} \min_{\bm{y} \in \B_{p,s} } c_{sa} + \lambda \cdot \sum_{s'=1}^{s-1} y_{sas'}v^{t+1}_{s'} + \lambda \cdot \sum_{s'=s}^{n} y_{sas'}v^{t}_{s'}.$

\paragraph{Anderson Value Iteration.} This algorithm \cite{ref-c}, inspired from quasi-Newton methods from convex optimization,  updates $\bm{v}^{t+1}$ as a linear combination of the last $(m+1)$-iterates $F(\bm{v}^{t}), ..., F(\bm{v}^{t-m})$:
\[ \bm{v}^{t+1} = \sum_{i=0}^{m} \alpha_{i} F(\bm{v}^{t-m+i}),\]
for some weights $\alpha_{0}, ..., \alpha_{m}$.
% We also consider Anderson VI (referred to as \textit{Anderson} in our figures), see \cite{ref-c}. In order to compute the next iterates $\bm{v}^{t+1}$, Anderson VI computes weights $\alpha_{0}, ..., \alpha_{m}$ and updates $\bm{v}^{t+1}$ as a linear combination of the last $(m+1)$-iterates $F(\bm{v}^{t}), ..., F(\bm{v}^{t-m})$:
%\[ \bm{v}^{t+1} = \sum_{i=0}^{m} \alpha_{i} F(\bm{v}^{t-m+i}).\]
The weights $\bm{\alpha} \in \R^{m+1}$ are updated at every iteration, see Algorithm 1 and Equation (1) in \cite{ref-c} for further details. There is no heuristics for choosing $m$; we choose $m=5$ in our numerical experiments.

\section{Details on machine replacement example}\label{app:machine}
We present an example of this instance in Figure \ref{fig:Machine-MDP-1}-\ref{fig:Machine-MDP-2}, where there are $10$ states: 8 states related to the condition of the machine, and two repair states. The instances for larger number of states are constructed in the same fashion by adding some condition states for the machine.  Below we give details about the states, actions, transitions and rewards.

\paragraph{States.} The machine replacement problem involves a machine whose set of possible conditions are described by $S$ states.  The first $S-2$ states are operative states.   The states $1$ to $S-2$ model the condition of the machine,  with $1$ being perfect condition and $S-2$ being worst condition.  The last two states $S-1$ and $S$ are states representing when the machine is being repaired.   The initial distribution is uniform across states. 

\paragraph{Actions.} There are two actions: \textit{repair} and \textit{no repair}.

\paragraph{Transitions.} The transitions are detailed in Figures \ref{fig:Machine-MDP-1}-\ref{fig:Machine-MDP-2}.  When the action is \textit{no repair}, the machine is likely to deteriorates toward the state $S-2$, or may stay in the same condition. When the action is \textit{repair}, the decision-maker brings the machine to the states $S-1$ and $S-2$.

\paragraph{Rewards.}There is a cost of 0 for states $1, ..., S-3$; letting the machine reach the worst operative state $S-2$ is penalized with a cost of $20$. 
The state $S-1$ is a standard repair state and has a cost of 2, while the last state $S$ is a longer and more costly repair state and has cost 10. 

\begin{figure}[H]
     \centering
         \includegraphics[scale=0.15]{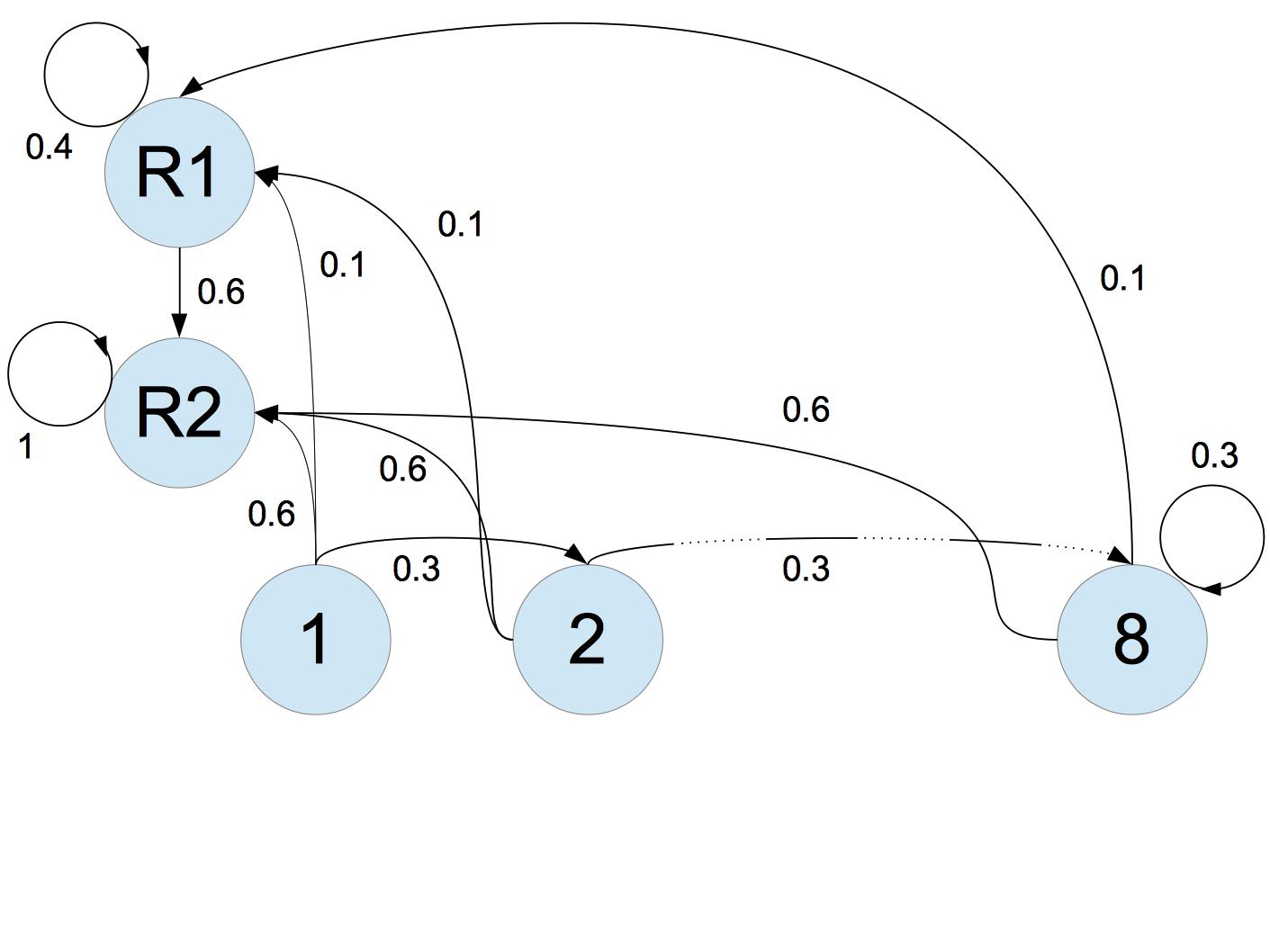}
         \caption{Nominal transition for action = \textit{ repair } in our machine replacement MDP.}
           \label{fig:Machine-MDP-1}
     \end{figure}
\begin{figure}[H]
         \centering
         \includegraphics[scale=0.15]{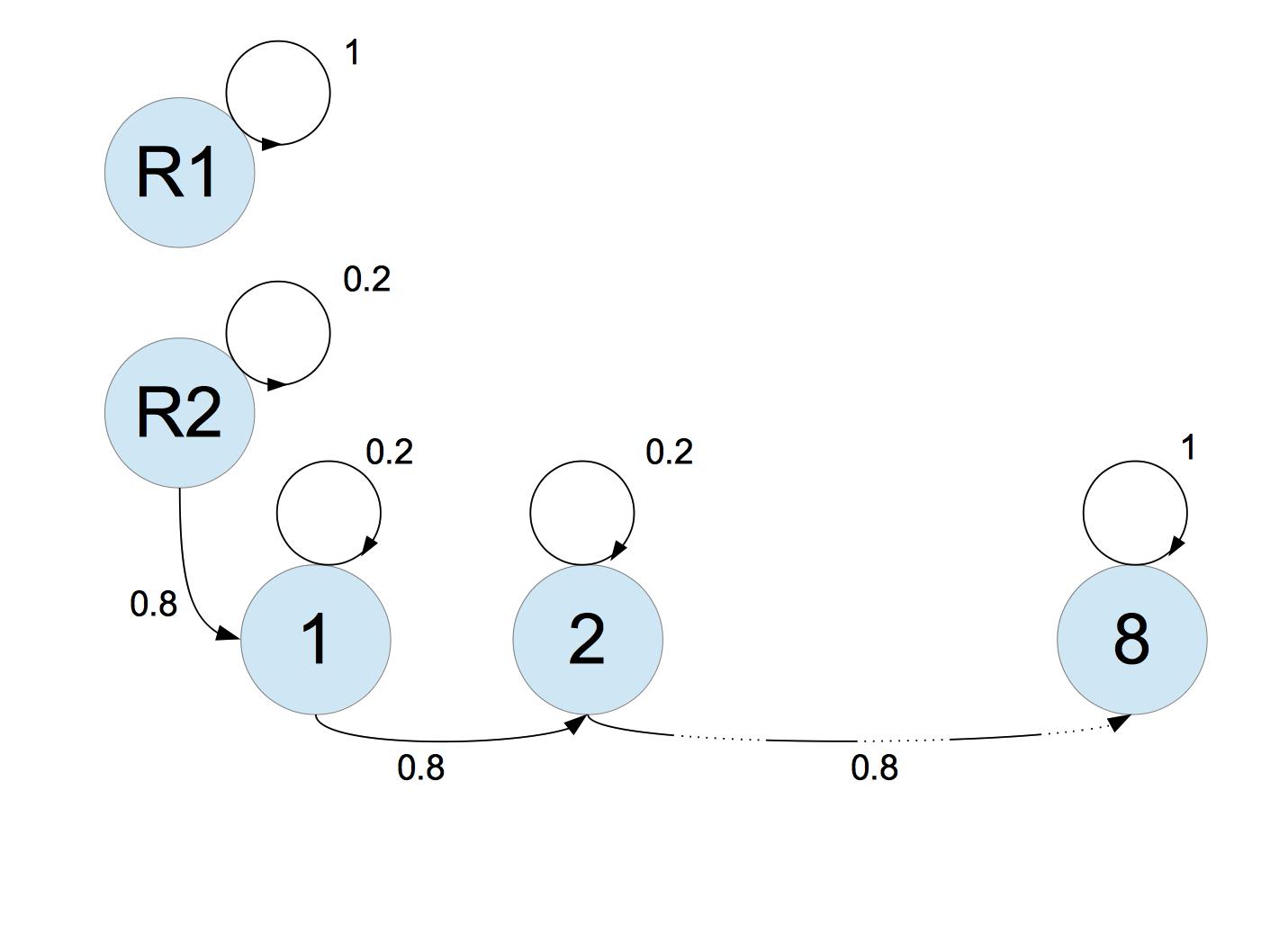}
         \caption{Nominal transition for action = \textit{ no repair } in our machine replacement MDP.}
         \label{fig:Machine-MDP-2}
\end{figure}

\section{Details on forest management example}\label{app:forest}
The state in the forest management example represents the growth of the forest. The goal is to find the right balance between maintaining the forest,  making money by selling cut wood. Every year, the forest may suffer from wildfires. A complete description may be found at \cite{pymdp}.  This is inspired from the application of dynamic programming to optimal fire management \cite{possingham1997application}.
\paragraph{States.} There are $S$ states. The state $1$ is the youngest state for the forest. The forest can not grow beyond state $S$. The initial distribution is uniform across states. 
\paragraph{Actions.} There are two actions, \textit{wait} and \textit{cut \& sell}.
\paragraph{Transitions. } If the forest is in a state $s$ and the action is \textit{wait}, the next state is $s+1$ with probability $1-p$ (the forest grows) and $1$ with probability $p$ (a wildfire burns the forest down). If the forest is in a state $s$ and the action is \textit{cut \& sell}, the next state is $1$ with probability $1$. The probability of wildfire $p$ is chosen at $p=0.1$.
\paragraph{Rewards.} There is a reward of $4$ when the forest reaches the oldest state ($S$) and the chosen action is \textit{wait}.  There is a reward of $0$ at every other state if the chosen action is \textit{wait}. When the action is \textit{cut \& sell}, the reward at the youngest state $s=1$ is $0$, there is a reward of $1$ in any other state $s \in \{1, ..., S-1\}$, and a reward of $2$ in $s=S$. We convert all rewards to cost by flipping the signs of the rewards. 

\end{document}